\documentclass[a4paper,10pt,twoside]{amsart}
\usepackage{hyperref}
\usepackage{amsfonts}
\usepackage{amssymb}
\usepackage{amsmath}
\usepackage[cmtip,all]{xy}
\usepackage{lscape}

\newtheorem{proposition}{Proposition}[section]
  \newtheorem{lemma}[proposition]{Lemma}
  \newtheorem{theorem}[proposition]{Theorem}
  \newtheorem{corollary}[proposition]{Corollary}
\theoremstyle{definition}
  \newtheorem{definition}[proposition]{Definition}
  \newtheorem*{definition*}{Definition}
  
  \newtheorem*{example*}{Example}
  \newtheorem{remark}[proposition]{Remark}
  \newtheorem*{noteadded}{Note added}

\numberwithin{equation}{section}

\def\co{c}
\def\id{\mathsf{id}}
\def\bc{\begin{center}}
\def\ec{\end{center}}

\begin{document}

\title{On Hopf monoids in duoidal categories}

\author{Gabriella B\"ohm}
\address{Wigner Research Centre for Physics, Budapest,
H-1525 Budapest 114, P.O.B.\ 49, Hungary}
\email{bohm.gabriella@wigner.mta.hu}
\author{Yuanyuan Chen}
\address{Department of Mathematics, Nanjing Agricultural University, Nanjing
210095, P.R. China.}
\author{Liangyun Zhang}
\address{Department of Mathematics, Nanjing Agricultural University, Nanjing
210095, P.R. China.}
\thanks{Y.Y. Chen participated in this research as a fellow of the Hungarian
Scholarship Board. Her work was supported also by the Innovative
Project of Jiangsu province for Graduate Cultivation. G. B\"ohm is
grateful to the other authors for a generous invitation and a warm hospitality 
at Nanjing Agricultural University in March 2012.}
\begin{abstract}
Aguiar and Mahajan's bimonoids $A$ in a duoidal category $\mathcal M$ are
studied. Under certain assumptions on $\mathcal M$, the Fundamental Theorem of
Hopf Modules is shown to hold for $A$ if and only if the unit of $A$ determines
an $A$-Galois extension. Our findings are applied to the particular
examples of small groupoids and of Hopf algebroids over a commutative base
algebra.
\end{abstract}
\date{May 2013}
\maketitle

\section*{Introduction}
There are several equivalent conditions on a bialgebra $A$ (say, over a
commutative ring $k$) under which we say that it is a Hopf algebra
(see e.g. the textbooks \cite{S. Montgomery, M. Sweedler, BrzWis}):
\begin{itemize}
\item[{(i)}] The identity map $A\to A$ has an inverse --- called the antipode
--- in the convolution algebra $\mathsf{End}(A)$.
\item[{(ii)}] $A$ induces a right Hopf monad $(-)\otimes A$ on the category
$\mathcal M_k$ of $k$-modules; in the sense of \cite{A. Bruguieres}. That is,
the closed structure of $\mathcal M_k$ is lifted to the category of right
$A$-modules.
\item[{(iii)}] $A$ is an $A$-Galois extension of $k$. That is, a canonical
comonad morphism is an isomorphism.
\item[{(iv)}] The Fundamental Theorem of Hopf Modules \cite{R. Larson}
holds. That is, the category of $A$-Hopf modules is equivalent to the category
of $k$-modules.
\end{itemize}

In their monograph \cite{M. Aguiar}, Marcelo Aguiar and Swapneel Mahajan
generalized bialgebras to bimonoids in so-called duoidal categories (termed
``2-monoidal'' in their work). These are categories equipped with two
different monoidal structures. They are required to be compatible in the sense
that the functors and natural transformations defining the first monoidal
structure, are opmonoidal with respect to the second monoidal
structure. Equivalently, the functors and natural transformations defining the
second monoidal structure, are monoidal with respect to the first monoidal
structure. More details will be recalled in Section \ref{sec:duo_cat}. A
bimonoid is a monoid with respect to the first monoidal structure and a
comonoid with respect to the second monoidal structure. The compatibility
axioms are formulated in terms of the coherence morphisms between the monoidal
structures. 

A natural question arises how to define a {\em Hopf} monoid in a duoidal
category. There are at least four possibilities listed above.

The first possibility (i) does not seem applicable. Since the monoid and
comonoid structures are defined in different monoidal categories, the notion
of convolution monoid is not available.
The second possibility (ii) was taken in \cite{T. Booker} and it was
investigated in relation with the lifting of closed structures.
The aim of this paper is to study the remaining possibilities (iii) and
(iv). Under certain assumptions on the duoidal category we work in, we prove
their equivalence; and that they hold whenever (ii) does.

Let us re-visit for a moment the classical case of a bialgebra $A$ over a
commutative ring $k$. Since a bialgebra carries both an algebra and a
coalgebra structure, it induces both a monad and a comonad on the category
$\mathcal M_k$ of $k$-modules (with respective Eilenberg-Moore categories
${\mathcal M}_A$, the category of $A$-modules; and ${\mathcal M}^A$, the
category of $A$-comodules). These are in turn related by a mixed distributive
law in the sense of \cite{J. Beck}. The category of its mixed modules
${\mathcal M}^A_A$ is known as the category of $A$-Hopf modules. Moreover,
there is an associated  triangle of functors
$$
\xymatrix @C=70pt @R=50pt{
&{\mathcal M}^A_A\ar@/_.6pc/[d]_-{U^A}\ar@{}[d]|-{\dashv}\\
\mathcal M_k\ar@/^.6pc/[r]^(.58){F_A=(-)\otimes A}
\ar@{}[r]|-{\perp}
\ar[ru]^-{K=(-)\otimes A}&
{\mathcal M}_A\ar@/_.6pc/[u]_-{F^A=(-)\otimes A}
\ar@/^.6pc/[l]^-{U_A}}
$$
in which $U^AK=F_A$ (and $U^A$ and $U_A$ denote forgetful functors).

The Dubuc-Beck theory \cite{J. Beck, E. Dubuc, J. Power, M. Barr} (shortly
recalled in Section \ref{sec:Dubuc-Beck}) tells us the sufficient and
necessary conditions under which $K$ is an equivalence --- that is, the
Fundamental Theorem of Hopf Modules holds. The first step is to see that $K$
possesses a right adjoint $N$. By Dubuc's Adjoint Lifting Theorem
\cite{E. Dubuc} (see the dual form of \cite[Theorem 2.1]{J. Power}), $N$ 
is given by certain equalizers which exist by the completeness of $\mathcal
M_k$. Then there is a canonical comonad morphism 
$$
\beta:=\big(
\xymatrix@C=40pt{
F_AU_A\cong
U^AK N F^A\ar[r]^-{U^A\hat\epsilon F^A}&
U^AF^A}\big),
$$
where $\hat \epsilon$ denotes the counit of the adjunction $K\dashv N$. With
its help, $K$ is an equivalence if and only if $\beta$ is a natural
isomorphism, and $F_A$ preserves the equalizers defining the right adjoint of 
$K$ and it reflects isomorphisms (see e.g. \cite{GomTor,Mes}).

By the above reasoning it is immediate that if $K$ is an equivalence ---
i.e. above condition (iv) defining a Hopf algebra holds --- then
$\beta$ is a natural isomorphism. That is, above condition (iii) defining
a Hopf algebra $A$ holds.
Conversely, assume that $\beta$ is a natural isomorphism ---
i.e. above condition (iii) defining a Hopf algebra holds. Then in this
particular situation associated to a $k$-bialgebra, the equalizers defining
the right adjoint of $K$ turn out to be contractible, hence absolute
equalizers (with contracting maps constructed in terms of $\beta^{-1}$). Such 
equalizers are preserved by any functor, in particular by $F_A.$ This allows
to prove also that $F_A$ reflects  isomorphisms, hence $K$ is an
equivalence. That is, above condition (iv) defining a Hopf algebra $A$ holds.
(Note that since $A$ is a generator in $\mathcal M_A$,
$\beta_M\equiv M\otimes_A \beta_A$ is an isomorphism for all right $A$-modules
$M$ if and only if $\beta_A$ is an isomorphism. The equivalence of this form of
condition (iii) and condition (iv) above, can be found e.g. in \cite[Section
15.5 (b)$\Leftrightarrow$(g)]{BrzWis}.) 

In this paper we apply a similar strategy in the case of a bimonoid $A$ in a
duoidal category $\mathcal M$. In Section \ref{sec:prelims} we recall 
the Dubuc-Beck theory in a nutshell (but with precise references) and 
some basic facts about duoidal categories and their bimonoids. In
Section \ref{sec:rel_Hopf} we study the analogue of the above adjoint triangle
and the corresponding comonad morphism $\beta$ in a somewhat more general
setting: We take an $A$-comodule monoid $B$ and we consider the category of
$(A,B)$-relative Hopf modules. The isomorphism property of the corresponding
comonad morphism $\beta$ defines the notion of $A$-Galois extension $C\to B$.
In Section \ref{sec:fthm} we restrict to (non-relative) $A$-Hopf
modules. We assume that idempotent morphisms in $\mathcal M$ split and that a
canonical functor $H$ --- between the category $\mathcal M^I$ of comodules over
one monoidal unit $I$, and the category $\mathcal M_J$ of modules over the
other monoidal unit $J$ in $\mathcal M$ --- is fully faithful. 
(In the case when $\mathcal M$ is a braided monoidal category, $H$ 
is a trivial isomorphism.) Under these assumptions we prove the Fundamental
Theorem of Hopf Modules which takes now the following form: The canonical
comparison functor from $\mathcal M^I$ to the category of $A$-Hopf modules is
an equivalence if and only if the canonical comonad morphism $\beta$ is a
natural isomorphism, and if and only if $I\to A$ is an $A$-Galois extension.
The assumptions made on a duoidal category in our Fundamental
Theorem of Hopf Modules, respectively, their dual counterparts, are shown to
hold in two duoidal categories described in \cite{M. Aguiar}: In the category
of spans (over a fixed base set) and in the category of bimodules (over a
fixed commutative, associative and unital algebra). 
So as an application of the theorem, we obtain that a small category
$A$ is a Galois extension of its object set $X$ if and only if the category of
$A$-Hopf modules is equivalent to the slice category
$\mathsf{set}/X$; and if and only if $A$ is a groupoid. Similarly, if
$A$ is a bialgebroid (called a ``$\times_R$-bialgebra''
in \cite{M. Takeuchi}) --- over a commutative algebra $R$ and
such that the unit $R\otimes R\to A$ takes its values in the center of $A$
---, then the category of $A$-Hopf modules is equivalent to the category of
$R$-modules if and only if $A$ is a Hopf algebroid (in the sense
of \cite{G. Bohm} and references therein; for the commutative case see 
also \cite{C. Ravenel}).

\section{Preliminaries}\label{sec:prelims}

\subsection{The Dubuc-Beck theory.}\label{sec:Dubuc-Beck}

For later application, in this section we consider the following
situation. Let $T$ and $S$ be comonads on respective categories $\mathcal
A$ and $\mathcal B$. Denote their Eilenberg-Moore categories of comodules
(also called coalgebras, see e.g. the dual of the notion in \cite[page
88]{M. Barr}) by $\mathcal A^T$ and $\mathcal B^S$, respectively, with
forgetful functors $U^T:\mathcal A^T\to \mathcal A$ and $U^S:\mathcal
B^S \to \mathcal B$. (These functors possess right adjoints, to be denoted by
$F^T$ and $F^S$, respectively; such that $T=U^TF^T$ and $S=U^SF^S$). Assume that
there is an adjunction $L \dashv R:\mathcal B \to \mathcal A$ (with unit
$\nu:\mathcal A \to RL$ and counit $\epsilon:LR \to \mathcal B$). In what
follows, we re-collect from the literature some results concerning sufficient
and necessary conditions for the existence of an {\em equivalence} $K:\mathcal
A^T \to \mathcal B^S$ rendering commutative 
\begin{equation}\label{eq:gen_lift} 
\xymatrix{
\mathcal A^T \ar[r]^-K\ar[d]_-{U^T}& 
\mathcal B^S \ar[d]^-{U^S}\\
\mathcal A \ar[r]_-L&
\mathcal B.}
\end{equation}

Whenever the diagram \eqref{eq:gen_lift} commutes for some functor $K$, $K$ is
said to be a {\em lifting} of $L$. By \cite[Corollary 5.11]{PowWat} (or
by \cite[Proposition 1.1 and Theorem 1.2]{GomTor}), the liftings $K$ of $L$
are in a bijective correspondence with the so-called {\em comonad morphisms}
of the form $(L,\lambda)$ (this result is usually attributed to D. Applegate's
thesis from 1965). The comonad morphisms \cite{R. Street} from $T$ to $S$ are
pairs consisting of a functor $L:\mathcal A \to \mathcal B$ and a natural
transformation $\lambda:LT \to SL$ which is compatible with the comonad
structures $(\delta^T:T \to T^2, \varepsilon^T:T \to \mathcal A)$ and
$(\delta^S:S \to S^2, \varepsilon^S:S \to \mathcal B)$ in the sense of the
commutative diagrams  
$$
\xymatrix{
LT\ar[rr]^-\lambda\ar[d]_-{L\delta^T}&&
SL\ar[d]^-{\delta^S L}&&
LT\ar[r]^-\lambda\ar[d]_-{L \varepsilon^T}&
SL \ar[d]^-{\varepsilon^S L}\\
LT^2\ar[r]_-{\lambda T}&
SLT \ar[r]_-{S\lambda}&
S^2 L&&
L\ar@{=}[r]&
L.}
$$ 

Let us assume that there exists such a comonad morphism $\lambda:LT\to SL$
hence a lifting $K$ of $L$. Then there is also an induced comonad morphism
from the comonad $LTR$ (with comultiplication 
$\xymatrix{
LTR\ar[r]^-{L\delta^T R}&
LT^2R\ar[r]^-{LT\nu TR}&
LTRLTR
}$ and counit
$\xymatrix@C=8pt{
LTR\ar[rr]^-{L\varepsilon^T R}&&
LR\ar[r]^-{\epsilon}&
\mathcal B}
$) to $S$, see \cite[Theorem II.1.1]{E. Dubuc} (and also \cite[Theorem
1.2]{GomTor}). It is given by the identity functor $\mathcal B$ and the natural
transformation 
\begin{equation}\label{eq:beta_gen}
\beta=(\xymatrix{
LTR \ar[r]^-{\lambda R}&
SLR\ar[r]^-{S\epsilon}&
S}).
\end{equation}
Alternatively, using the correspondence between $\lambda$ and $K$, it can be
re-written as 
\begin{equation}\label{eq:beta_long}
\beta=(\xymatrix{
LTR=U^SKF^TR \ar[rr]^-{U^S\nu^SKF^TR}&&
U^SF^SU^SKF^TR=SLTR\ar[r]^-{SL\varepsilon^TR}&
SLR\ar[r]^-{S\epsilon}&
S}),
\end{equation}
where $\nu^S:\mathcal B^S\to F^SU^S$ is the unit of the adjunction $U^S\dashv
F^S$. 

Since at the end we want $K$ in \eqref{eq:gen_lift} to be an equivalence, it
should possess in particular a right adjoint. For the existence of this right
adjoint, we obtain sufficient and necessary conditions from Dubuc's Adjoint
Lifting Theorem. We apply it in the form which is dual to \cite[Theorem
2.1]{J. Power}. Re-draw \eqref{eq:gen_lift} in the form
\begin{equation}\label{eq:lift_tr}
\xymatrix{
&& \mathcal B^S\ar[d]^-{U^S}\\
\mathcal A^T\ar[r]_-{U^T}
\ar[rru]^-K&
\mathcal A\ar[r]_-L&
\mathcal B.}
\end{equation}
In the triangular diagram \eqref{eq:lift_tr}, both the functor in the bottom
row and that in the right vertical are left adjoints. Moreover, $U^S$ is
comonadic so in particular of the co-descent type. Hence by the dual form
of \cite[Theorem 2.1]{J. Power} (see also \cite[Theorem A.1]{E. Dubuc}
and \cite[Proposition 1.3]{GomTor}), $K$ possesses a right adjoint $N$ if and
only if for every $S$-comodule $(B,\rho:B\to SB)$, there exists the equalizer 
$$
\xymatrix{
N(B,\rho)\ar[r]&
TRB\ar@<2pt>[rrrr]^-{TR\rho}
\ar@<-2pt>[rrrr]_-{TRS\epsilon B.TR\lambda RB.T\nu TRB.\delta^T RB}&&&&
TRSB} \qquad \textrm{in}\ \mathcal A^T,
$$
providing the object map of $N$. By the uniqueness of the adjoint up-to
natural isomorphism, whenever the right adjoint $N$ of $K$ exists it obeys $N
F^S\cong F^T R$, and the counit $\hat \epsilon:KN \to \mathcal B^S$ of the
adjunction $K\dashv N$ renders commutative
$$
\xymatrix{
LTR\ar[d]_-\cong \ar[r]^-{L\varepsilon^T R}&
LR \ar[r]^-\epsilon&
\mathcal B \ar@{=}[d]\\
U^S KN F^S\ar[r]_-{U^S \hat \epsilon F^S}&
S \ar[r]_-{\varepsilon^S}&
\mathcal B.}
$$
Using this identity together with \eqref{eq:beta_long} (and with one of
the triangle identities on the adjunction $U^S\dashv F^S$), we can re-write
\eqref{eq:beta_gen} in the alternative form   
\begin{equation}\label{eq:beta_short}
\beta=(\xymatrix{
LTR\ar[r]^-\cong&
U^S KNF^S \ar[r]^-{U^S \hat \epsilon F^S}&
S})
\end{equation}
whenever the right adjoint $N$ of $K$ exists.

Finally, for any lifting $K$ of $L$ in \eqref{eq:gen_lift}, the following
assertions are equivalent (see e.g. \cite[Theorem 1.7]{GomTor}
or \cite[Theorem 4.4]{Mes}).  
\begin{itemize}
\item The functor $K$ is an equivalence.
\item The natural transformation \eqref{eq:beta_gen} is an isomorphism and 
$\xymatrix@C=12pt{
\mathcal A^T\ar[r]^-{U^T}&
\mathcal A \ar[r]^-L&
\mathcal B}$
is comonadic.
\end{itemize}
Recall that by the dual form of Beck's monadicity theorem (see e.g. \cite[page
100, Theorem 3.14]{M. Barr}), a left adjoint functor $L$ is comonadic (or
co-tripleable) if and only if it reflects isomorphisms and creates the
equalizers of $L$-contractible equalizer pairs.  

\subsection{Duoidal category.}\label{sec:duo_cat}

In this section we recall from \cite{M. Aguiar} some information about
so-called duoidal (also known as 2-monoidal) categories. These are categories
equipped with two related monoidal structures. Bimonoids in duoidal categories
and their induced bimonads are also recalled.

\begin{definition} \cite[Definition 6.1]{M. Aguiar}
A {\it duoidal category} is a category $\mathcal{M}$ equipped with two
monoidal products $\circ$ and $\bullet$ with respective units $I$ and
$J$, along with morphisms
$$
\xymatrix{I\ar[r]^-{\delta} &I\bullet I}, \ \
\xymatrix{J\circ J\ar[r]^-{\varpi}&J}, \ \
\xymatrix{I\ar[r]^-\tau &J},
$$
and, for all objects $A,B,C,D$ in $\mathcal M$, a morphism
$$
\zeta_{A,B,C,D}:(A\bullet B)\circ(C\bullet D)\rightarrow (A\circ
C)\bullet(B\circ D),
$$
called the {\it interchange law}, which is natural in all of the four
occurring objects. These morphisms are required to obey the axioms
below.

\textbf{Compatibility of the units.} The monoidal units $I$ and
$J$ are compatible in the sense that
\begin{center}
$(J,\varpi,\tau)$ is a monoid in $(\mathcal{M},\circ,I)$ and
$(I,\delta,\tau)$ is a comonoid in $(\mathcal{M},\bullet,J)$.
\end{center}

\textbf{Associativity.} For all objects $A,B,C,D,E,F$ in $\mathcal
M$, the following diagrams commute.
\begin{equation} \label{eq1.1}
\xymatrix{
((A\bullet B)\circ(C\bullet D))\circ(E\bullet F)
\ar[r]^{\cong}\ar[d]_{\zeta\circ \id}
& (A\bullet B)\circ((C\bullet D)\circ(E\bullet F))\ar[d]^{\id\circ\zeta}\\
((A\circ C)\bullet(B\circ D))\circ(E\bullet F)\ar[d]_{\zeta}
& (A\bullet B)\circ((C\circ E)\bullet(D\circ F))\ar[d]^{\zeta}\\
((A\circ C)\circ E)\bullet((B\circ D)\circ F)\ar[r]_{\cong} & (A\circ (C\circ
E))\bullet(B\circ(D\circ F))\\
((A\bullet B)\bullet C)\circ((D\bullet E)\bullet F)\ar[r]^{\cong}\ar[d]_{\zeta}
& (A\bullet (B\bullet C))\circ(D\bullet (E\bullet F))\ar[d]^{\zeta}\\
((A\bullet B)\circ(D\bullet E))\bullet(C\circ F)\ar[d]_{\zeta\bullet \id}
& (A\circ D)\bullet((B\bullet C)\circ(E\bullet F))\ar[d]^{\id\bullet\zeta}\\
((A\circ D)\bullet(B\circ E))\bullet (C \circ F)\ar[r]_{\cong}
&(A\circ D)\bullet((B\circ E)\bullet (C \circ F))}
\end{equation}

\textbf{Unitality.} For all objects $A,B$ in $\mathcal M$, the
following diagrams commute.
\begin{equation}\label{eq1.2}
\xymatrix{
I\circ(A\bullet B)   \ar[r]^-{\delta\circ \id}\ar[d]_-{\cong}
& (I\bullet I)\circ(A\bullet B)\ar[d]^-{\zeta} &
(A\bullet B)\circ I  \ar[r]^-{\id\circ \delta}\ar[d]_-{\cong}
&  (A\bullet B)\circ(I\bullet I)\ar[d]^-{\zeta}\\
A\bullet B    & (I\circ A)\bullet(I\circ B) \ar[l]^-{\cong}
&A\bullet B   & (A\circ I)\bullet(B\circ I ) \ar[l]^-{\cong}\\
(J\circ J)\bullet(A\circ B)   \ar[r]^-{\varpi\bullet \id}
& J\bullet(A\circ B)\ar[d]^-{\cong}
&(A\circ B)\bullet (J\circ J)  \ar[r]^-{\id\bullet \varpi}
& (A\circ B)\bullet J \ar[d]^-{\cong}\\
(J\bullet A)\circ(J\bullet B)   \ar[r]_-{\cong}\ar[u]^-{\zeta}
&A\circ B
&(A\bullet J)\circ (B\bullet J)\ar[r]_-\cong \ar[u]^-{\zeta}&
A\circ B}
\end{equation}
The arrows labelled by $\cong$ in the diagrams above, refer to the
associativity and the unit constraints in either monoidal category. The
same notation is used in all diagrams throughout the paper. In the formulae,
however, we denote the associatior by $\alpha$ (the unitors do not happen to
occur). 
\end{definition}

By one of the unitality axioms in \eqref{eq1.2} and unitality of the monoid
$J$, also 
\begin{equation}\label{eq:M68}
\xymatrix@C=35pt@R=15pt{
(A\bullet I)\circ (B\bullet J) \ar[r]^-\zeta\ar[d]_-\cong&
(A\circ B)\bullet (I\circ J)\ar[d]^-\cong\\
(A \bullet I)\circ B\ar[d]_-{(A\bullet \tau)\circ B}&
(A\circ B)\bullet J\ar[d]^-\cong\\
(A\bullet J)\circ B\ar[r]_-\cong&
A\circ B}
\end{equation}
and some of its symmetrical variants commute, for any objects $A$ and
$B$ of $\mathcal M$, see \cite[Proposition 6.8]{M. Aguiar}.

The simplest examples of duoidal categories are braided monoidal
categories. In this case, both monoidal products coincide and the interchange
law is induced by the braiding, see \cite[Section 6.3]{M. Aguiar}.
Generalizing bimonoids in braided monoidal categories, bimonoids can be
defined also in duoidal categories --- as monoids in the category of
comonoids in $(\mathcal M, \bullet, J)$, equivalently, as comonoids in the
category of monoids in $(\mathcal M, \circ, I)$.  
Explicitly, this means the following. 

\begin{definition}\cite[Definition 6.25]{M. Aguiar}\label{Def:bimonoid}
A {\it bimonoid} in a duoidal category $\mathcal{M}$
is an object $A$ equipped with a monoid structure $(A\circ A \stackrel \mu \to
A,I\stackrel \eta\to A)$ in $(\mathcal{M},\circ, I)$, and a comonoid structure
$(A \stackrel \Delta\to A \bullet 
A,A\stackrel \varepsilon\to J)$ in $(\mathcal{M},\bullet,J)$,
subject to the compatibility axioms
$$
\xymatrix @C=11pt{
(A\bullet A)\circ(A\bullet A)\ar[rr]^-{\zeta}
&& (A\circ A)\bullet (A\circ A)\ar[d]_-{\mu\bullet\mu}
&
A\circ A \ar[r]^-{\varepsilon\circ\varepsilon}\ar[d]^-{\mu}
& J\circ J\ar[d]_-{\varpi}
&
I \ar[r]^-{\eta}\ar[d]^-{\delta}
& A\ar[d]_-{\Delta}
&
I\ar[rd]^{\tau}\ar[d]^-{\eta}\\
A\circ A\ar[u]_-{\Delta\circ\Delta}\ar[r]_-{\mu}
&  A \ar[r]_-{\Delta}
&  A\bullet A
&
A \ar[r]_-{\varepsilon}
&  J
&
I\bullet I \ar[r]_-{\eta\bullet\eta}
&  A\bullet A
&A\ar[r]_-{\varepsilon}
&  J.}
$$
\end{definition}

Note that in particular the monoidal units $I$ and $J$ are bimonoids in
any duoidal category. 

By modules over a bimonoid $A$ in a duoidal category $\mathcal M$,
modules over the constituent monoid $A$ in $(\mathcal M, \circ, I)$ are meant.
It was observed in \cite[Section 6.6]{M. Aguiar} that the category of
$A$-modules is monoidal with respect to the monoidal product $\bullet$ and
monoidal unit $J$. This fact can be given the following equivalent formulation.

\begin{proposition}  \label{prop:bimonad}
\cite[Theorem 18]{T. Booker}
Any bimonoid $A$ in a duoidal category $\mathcal{M}$ induces a bimonad
(termed a ``Hopf monad'' in \cite{I. Moerdijk}) $(-)\circ A$
on $(\mathcal{M},\bullet,J)$.
\end{proposition}

\begin{proof}
For later reference, we recall the forms of the structure morphisms of
the bimonad in the claim. The multiplication and unit of the monad are induced
by the multiplication and the unit of the monoid $A$, respectively. The binary
part of the opmonoidal structure is given by
$$
\xymatrix{
(M\bullet M')\circ A \ar[rr]^-{(M\bullet M')\circ\Delta}
&& (M\bullet M')\circ(A\bullet A)\ar[r]^-{\zeta}
& (M\circ A)\bullet (M'\circ A),
}
$$
for all objects $M,M'$ in $\mathcal M$. The nullary part is provided by
$
\xymatrix@C=15pt{
J\circ A \ar[r]^-{J\circ\varepsilon}
&  J\circ J \ar[r]^-{\varpi}
& J }.
$
\end{proof}

Following the terminology of \cite{A. Bruguieres}, a bimonad $T$ --- on
a monoidal category $\mathcal M$ with monoidal product $\otimes$ --- is 
called a {\em right Hopf monad} whenever
$$
\xymatrix @C=20pt
{ T(TM\otimes M')\ar[r]^-{T_2}
& T^{2}M\otimes TM'\ar[rr]^-{\mu M\otimes TM'}
&& TM\otimes TM' }
$$
is a natural isomorphism, equivalently, by \cite[Theorem 2.15]{A. Bruguieres},  
\begin{equation}
\beta_{Q,M'}:=\big(
\xymatrix{
T(Q\otimes M') \ar[r]^-{T_2}
& TQ\otimes TM' \ar[r]^-{\gamma\otimes TM'}
& Q\otimes TM' } \big)
\label{eq1.5}
\end{equation}
is a natural isomorphism (where $\mu$ is the multiplication of the monad
$T$, $T_2$ is the binary part of the opmonoidal structure, $M,M'$ are
objects in $\mathcal{M}$ and $(Q,\gamma)$ is an object in the Eilenberg-Moore
category $\mathcal{M}^T$). 

For a bimonoid $A$ in a duoidal category $\mathcal{M}$, consider the
induced bimonad $(-)\circ A$ in Proposition \ref{prop:bimonad}. The
corresponding canonical morphism (\ref{eq1.5}) takes the explicit form
\begin{eqnarray}
\label{eq1.6}&&\\
\xymatrix@C=15pt{
(Q\bullet M')\circ A \ar[rr]^-{(Q\bullet M')\circ\Delta}
&& (Q\bullet M')\circ(A\bullet A) \ar[r]^-{\zeta}
&  (Q\circ A)\bullet(M'\circ A)\ar[rr]^-{\gamma\bullet(M'\circ A)}
&& Q\bullet(M'\circ A). }
\nonumber
\end{eqnarray}

\section{Relative Hopf modules}\label{sec:rel_Hopf}

The aim of this section is to develop the notion of Galois extension by
a bimonoid in a duoidal category. This requires several steps. As in the case
of bialgebras --- over a field or, more generally, in a braided monoidal
category --- we start with defining a comodule monoid over a bimonoid $A$;
this is a monoid $B$ in the monoidal category of $A$-comodules. Relative Hopf
modules are then $B$-modules in the category of $A$-comodules. A comodule
monoid is shown to induce a functor to the category of relative Hopf
modules. Whenever this functor possesses a right adjoint, we can regard 
this right adjoint as the functor taking the `coinvariant part' of relative
Hopf modules. In particular, we can consider the coinvariant part $B^{\co}$ of
$B$ itself. It turns out that it admits a monoid structure and a monoid
morphism to $B$. This defines the notion of an $A$-extension $B^{\co} \to
B$. Finally, assuming that certain coequalizers exist and are preserved, 
we associate to any $A$-extension an adjoint triangle. Whenever the
corresponding comonad morphism is iso, we say that the $A$-extension in
question is a Galois extension. 

\subsection{Comodule monoids and relative Hopf modules}

For a bimonoid $A$ in a duoidal category $\mathcal M$, we denote by
${\mathcal M}^A$ the category of $A$-comodules; that is, the category of
comodules over the constituent comonoid $A$ in $(\mathcal M, \bullet,
J)$. Recall from \cite[Section 6.6]{M. Aguiar}, that ${\mathcal M}^A$ is a
monoidal category via the monoidal product $\circ$ and the monoidal unit
$I$. That is to say, the forgetful functor ${\mathcal M}^A \to \mathcal M$ is
strict monoidal.

\begin{definition}
A {\it right comodule monoid} over bimonoid $A$ in a duoidal category
$\mathcal M$, is a monoid $B$ in ${\mathcal M}^A$. Explicitly,
this means that there is a coassociative and counital coaction
$\rho:B\rightarrow B\bullet A$ and an associative multiplication $\mu:B\circ
B \to B$ with unit $\eta:I\to B$ such that the following diagrams commute.
\begin{equation}
\xymatrix@C=20pt{
B\circ B \ar[r]^-{\mu} \ar[d]_-{\rho\circ\rho}  &
B \ar[r]^-{\rho}&
B\bullet A
&
I \ar[rr]^-{\eta} \ar[d]_-{\delta}
&&  B \ar[d]^-{\rho}\\
(B\bullet A)\circ(B\bullet A) \ar[rr]_-{\zeta}
&&  (B\circ B)\bullet (A\circ A)\ar[u]_-{\mu\circ\mu}
&
I\bullet I \ar[rr] _-{\eta\bullet \eta}
&&  B\bullet A}
\label{eq2.1}
\end{equation}
\end{definition}

Any bimonoid $A$ is a comodule monoid over itself, via the coaction provided
by the comultiplication. The multiplication and the unit of $A$ are
$A$-comodule morphisms by the first and by the third axiom in
Definition \ref{Def:bimonoid}, respectively.

Since a right comodule monoid $B$ over a bimonoid $A$ in a duoidal
category $\mathcal M$ is defined as a monoid in $({\mathcal M}^A,\circ,I)$,
it induces a monad $(-)\circ B$ on ${\mathcal M}^A$ (lifted from the
monad $(-)\circ B$ on ${\mathcal M}$). Equivalently, the
comonad $(-)\bullet A$ on $\mathcal M$ lifts to a comonad on the category
${\mathcal M}_B$ of right $B$-modules. These liftings correspond to the
mixed distributive law (in the sense of \cite{J. Beck})
$$
\xymatrix @C=15pt{
(M\bullet A)\circ B\ar[rr]^-{(M\bullet A)\circ \rho}
&& (M\bullet A)\circ(B\bullet A)\ar[r]^-{\zeta}
& (M\circ B)\bullet (A\circ A) \ar[rr]^-{(M\circ B)\bullet\mu}
&&(M\circ B)\bullet A,}
$$
for any object $M$ of $\mathcal M$. (For more on the connection between
distributive laws and liftings, we refer to \cite{PowWat}.) 

\begin{definition}\label{def:Hopf_mod}
Consider a bimonoid $A$ in a duoidal category $\mathcal M$ and a right
$A$-comodule monoid $B$. By (right-right) {\em $(A,B)$-relative Hopf modules}
we mean modules for the monad $(-)\circ B$ on ${\mathcal M}^A$; equivalently,
comodules over the comonad $(-)\bullet A$ on ${\mathcal M}_B$.
Explicitly, this means a triple $(X,\gamma:X\circ B \to X,\rho:X\to
X\bullet A)$, where $\gamma$ is an associative and unital action, $\rho$ is a
coassociative and counital coaction such that the following compatibility
condition holds.
\begin{equation}
\xymatrix{
X\circ B \ar[r]^{\gamma} \ar[d]_-{\rho\circ\rho}
&  X  \ar[r]^-{\rho}
& X\bullet A  \\
(X\bullet A)\circ(B\bullet A) \ar[rr]_-{\zeta}
&&  (X\circ B)\bullet(A\circ A) \ar[u]_-{\gamma\bullet \mu}}
\label{eq2.2}
\end{equation}
Morphisms of $(A,B)$-relative Hopf modules are morphisms of both
$A$-comodules and $B$-modules. The category of right-right $(A,B)$-relative
Hopf modules is denoted by $\mathcal{M}^A_B$.
\end{definition}

Clearly, an $A$-comodule monoid $B$ is itself an $(A,B)$-relative
Hopf module via the action provided by the multiplication.
The compatibility between this action and the $A$-coaction on $B$ holds by the
requirement that the multiplication in $B$ is a morphism of
$A$-comodules.

In the case when $B$ is equal to $A$, the resulting category
$\mathcal{M}_A ^{A}$ is called the category of {\em $A$-Hopf modules}.

\subsection{Coinvariants}\label{sec:coinv}

If $A$ is a bialgebra over a field --- or more generally, a bimonoid in
a braided monoidal category $\mathcal M$ --- and $B$ is a right $A$-comodule
monoid, then there is a lifting of the functor $(-) \otimes B: \mathcal M \to
{\mathcal M}_B$ to $\mathcal M \to {\mathcal M}^A_B$. Whenever appropriate
equalizers in $\mathcal M$ exist, the lifted functor possesses a right adjoint
known as the `$A$-coinvariants' functor. In what follows, we look for the
analogue of this adjunction in the duoidal setting.

As the first crucial difference from the classical case, the functor
$(-) \circ B: \mathcal M \to {\mathcal M}_B$ will be lifted now to ${\mathcal
M}^I \to {\mathcal M}^A_B$ (which means a lifting to $\mathcal M \to {\mathcal
M}^A_B$, up-to isomorphism, if $\mathcal M$ is a braided monoidal category).

\begin{proposition}\label{prop:rel_Hopf_lifting}
Consider a bimonoid $A$ in a duoidal category $\mathcal M$ and a right
$A$-comodule monoid $B$.
Then the evident functor $(-)\circ B:\mathcal M \to {\mathcal M}_B$
has a lifting to ${\mathcal M}^I \to {\mathcal M}^A_B$.
\end{proposition}

\begin{proof}
The relevant comonad morphism (in the sense of \cite{R. Street}, 
cf. \cite[Corollary 5.11]{PowWat}, \cite[Proposition 1.1 and
Theorem 1.2]{GomTor} or Section \ref{sec:Dubuc-Beck} in the current
paper) is given by
\begin{equation}\label{eq:lambda^0}
\lambda^0_M:
\xymatrix{
(M\bullet I)\circ B\ar[r]^-{(M\bullet I)\circ \rho}
& (M\bullet I)\circ(B\bullet A)\ar[r]^-{\zeta}
& (M\circ B)\bullet(I\circ A)\ar[r]^-{\cong}
& (M\circ B)\bullet A,}
\end{equation}
for any object $M$ of $\mathcal M$.
It is evidently natural in $M$. 
It follows by the first identity in (\ref{eq2.1}), by one
of the associativity axioms in \eqref{eq1.1} and by naturality of
$\zeta$ and of the associativity and unit constraints that $\lambda_M^0$ is a
morphism of $B$-modules. 
Comultiplicativity and counitality of $\lambda^0$; that is, commutativity
of the diagrams
$$
\scalebox{.98}{
\xymatrix @C=15pt @R=10pt{
(M\bullet I)\circ B \ar[rr]^{\lambda^0_M}\ar[d]^-{(M\bullet \delta)\circ B}  &&
(M\circ B)\bullet A\ar[d]_-{(M\circ B)\bullet\Delta}
&
(M\bullet I)\circ B \ar[r]^-{\lambda^0_M} \ar[d]^(.4){(M\bullet \tau)\circ B}
& (M\circ B)\bullet A \ar[d]_(.6){(M\circ B)\bullet \varepsilon}
\\
(M\bullet (I\bullet I)) \circ B \ar[d]^-{\cong}
&& (M\circ B)\bullet(A\bullet A)\ar[d]_-{\cong}
&
(M\bullet J)\circ B \ar[d]^-{\cong}
& (M\circ B)\bullet J \ar[d]_-{\cong}\\
((M\bullet I)\bullet I)\circ B \ar[r]_-{\lambda^0_{M\bullet I}}
& ((M\bullet I)\circ B)\bullet A\ar[r]_-{\lambda^0_M\bullet A}
& ((M\circ B)\bullet A)\bullet A
&
M\circ B \ar@{=}[r]
&  M\circ B}}
$$
follows by coassociativity and counitality of $\rho$, unitality of the monoid
$(J,\varpi,\tau)$, naturality of $\zeta$ and of the unit constraints, one of
the associativity axioms, see \eqref{eq1.1}, and two of the unitality axioms
in a duoidal category, see \eqref{eq1.2}. 
\end{proof}

For any right $I$-comodule $Z$, the $A$-coaction on $Z\circ B$ induced
by the comonad morphism \eqref{eq:lambda^0} is 
$$
\xymatrix@C=20pt{
Z\circ B\ar[r]^-{\rho\circ B}&
(Z\bullet I)\circ B\ar[r]^-{\lambda^0_Z}&
(Z\circ B)\bullet A.}
$$ 

Next we apply the Adjoint Lifting Theorem (in the form which is dual to
\cite[Theorem 2.1]{J. Power}) to construct the right adjoint of the lifted
functor in Proposition \ref{prop:rel_Hopf_lifting}. In the diagram 
\begin{equation}
\xymatrix{
\mathcal{M}^I \ar[rr]^-{(-)\circ B} \ar[d]_{U^I}
&&  \mathcal{M}^A_B \ar[d]^-{U^A} \\
\mathcal{M}  \ar[rr]_-{(-)\circ B}
&&  \mathcal{M}_B  }
\label{eq2.3}
\end{equation}
$U^I$ and $U^A$ are forgetful functors of comonads and the functor in the
bottom row is a left adjoint of a forgetful functor of a monad. Hence all of
them possess right adjoints, so the natural transformation $\lambda^0$
in \eqref{eq:lambda^0} has a mate
$$
\overline{\lambda}_Q = \big(
\xymatrix@C=20pt{
Q\bullet I \ar[r]^-\cong
& (Q\bullet I )\circ I  \ar[rr]^-{(Q\bullet I )\circ \eta}
&& (Q\bullet I)\circ B\ar[r]^{\lambda^0_Q}
& (Q\circ B)\bullet A\ar[r]^-{\gamma\bullet A}
& Q\bullet A} \big),
$$
for any object $(Q,\gamma)$ of ${\mathcal M}_B$.
In fact, $\overline{\lambda}_Q=Q\bullet\eta$ as the following computation
shows.
$$
\xymatrix @R=10pt{
Q\bullet I\ar[r]^-{\cong} \ar@{=}@/_2pc/[ddddr]
&  (Q\bullet I)\circ I \ar[rr]^-{(Q\bullet I)\circ \eta}
\ar[d]_-{(Q\bullet I)\circ\delta}
&&  (Q\bullet I)\circ B\ar[d]^-{(Q\bullet I)\circ\rho} \\
& (Q\bullet I)\circ (I\bullet I)
\ar[rr]^-{(Q\bullet I)\circ (\eta\bullet\eta)}\ar[d]_-{\zeta}
&& (Q\bullet I)\circ(B\bullet A)\ar[d]^-{\zeta} \\
& (Q\circ I)\bullet(I\circ I)
\ar[rr]^-{(Q\circ\eta)\bullet(I\circ\eta)}\ar[dd]_-{\cong}
&&  (Q\circ B)\bullet(I\circ A)\ar[d]^-{\cong} \\
&
&& (Q\circ B)\bullet A \ar[d]^-{\gamma\bullet A}  \\
&Q\bullet I\ar[rr]_-{Q\bullet\eta}
&& Q\bullet A }
$$
The top right square commutes since $\eta:I\to B$ is a morphism of
$A$-comodules. The region below it commutes by the naturality of $\zeta$. The
region at the bottom right commutes by unitality of the $A$-action on
$B$. Finally, the region on the left commutes by one of the unitality axioms
of a duoidal category, cf.\eqref{eq1.2}. 

Applying the Adjoint Lifting Theorem (see the dual form of \cite[Theorem
2.1]{J. Power} and Section \ref{sec:Dubuc-Beck}), we conclude that the
functor $(-)\circ B:{\mathcal M}^I \to {\mathcal M}^A_B$ in Proposition
\ref{prop:rel_Hopf_lifting} possesses a right adjoint $(-)^{\co}:{\mathcal
M}^A_B \to {\mathcal M}^I$ if and only if the equalizer
\begin{equation}\label{eq:McoA}
\xymatrix @C=30pt{
X^{\co} \ar[r]^-{\iota} & X\bullet I
\ar@<2pt>[rrr]^-{\varphi^0:=\rho\bullet I}
\ar@<-2pt>[rrr]_-{\varphi^1:=((X\bullet \eta)\bullet I).\alpha^{-1}.
(X\bullet\delta)}
&&&  (X\bullet A)\bullet I
}
\end{equation}
exists in $\mathcal{M}^I$ for any object $(X,\gamma,\rho)$ in
$\mathcal{M}^A_B$. 
We call $X^{\co}$ the {\em $A$-coinvariant part} of $X$.

\subsection{The submonoid of coinvariants}
If  the equalizers (\ref{eq:McoA}) exist in $\mathcal{M}^I$ then, in
particular, for any right comodule monoid $B$ over a bimonoid $A$, there is an
equalizer 
\begin{equation}
\xymatrix @C=30pt{
B^{\co} \ar[r]^-{\iota}
& B\bullet I \ar@<2pt>[rrr]^-{\varphi^0:=\rho\bullet I}
\ar@<-2pt>[rrr]_-{\varphi^1:=((B\bullet \eta)\bullet I).\alpha^{-1}.
(B\bullet\delta)} &&&
(B\bullet A)\bullet I}
\label{eq:BcoA}
\end{equation}
in $\mathcal{M}^I$. 

\begin{proposition}\label{prop:BcoA_monoid}
For any right comodule monoid $B$ over a bimonoid $A$ in a duoidal category
$\mathcal M$, the coinvariant part $B^{\co}$ --- whenever it exists --- is a
monoid in ${\mathcal M}^I$.
\end{proposition}

\begin{proof}
First we claim that $B\bullet I$ is a monoid in ${\mathcal M}^I$.
Since $B$ is a monoid in $\mathcal{M}^A$ by definition, and the
forgetful functor ${\mathcal M}^A\to {\mathcal M}$ is strict
monoidal, it follows that $B$ is a monoid in $\mathcal{M}$.
Since $I$ is a bimonoid in $\mathcal M$, also the forgetful functor ${\mathcal
M}^I\to {\mathcal M}$ is strict monoidal so in particular opmonoidal.
Since the right adjoint of an opmonoidal functor is
monoidal, $(-)\bullet I: \mathcal M\to {\mathcal M}^I$ is monoidal. Thus
it takes the monoid $B$ in $\mathcal M$ to the monoid $B\bullet I$ in
$\mathcal{M}^I$.

Both morphisms
$$
\xymatrix @R=10pt @C=15pt{
\big(B^{\co}\circ B^{\co}\ar[r]^-{\iota\circ\iota}
& (B\bullet I)\circ(B\bullet I)\ar[r]^-{\mu}
&B\bullet I \big)=\\
\big(B^{\co}\circ B^{\co}\ar[r]^-{\iota\circ\iota}
& (B\bullet I)\circ(B\bullet I)\ar[r]^-{\zeta}
& (B\circ B)\bullet(I\circ I) \ar[r]^-{\cong}
& (B\circ B)\bullet I \ar[r]^-{\mu\bullet I} & B\bullet I\big) }
$$
and
$
\xymatrix @C=15pt{
\big(I\big.\ar[r]^-\eta
&\big. B\bullet I \big)=
\big(I \big.\ar[r]^-\delta
&I\bullet I\ar[r]^-{\eta\bullet I}
&\big. B\bullet I\big)}
$
equalize the parallel morphisms $\varphi^0$ and $\varphi^1$ in the equalizer
(\ref{eq:BcoA}). Hence we obtain the multiplication and unit of $B^{\co}$ by
universality.
\end{proof}

\begin{proposition} \label{prop:omega}
Let $B$ be a right comodule monoid over a bimonoid $A$ whose coinvariant part
$B^{\co}$ exists. Then 
$$
\omega=\big(
\xymatrix{
B^{\co} \ar[r]^-{\iota}
&  B\bullet I\ar[r]^-{B\bullet\tau}
& B\bullet J\ar[r]^-{\cong} &
B}\big)
$$
is a morphism of monoids.
\end{proposition}

\begin{proof}
The monomorphism $\iota:B^{\co}\to B\bullet I$ is a morphism of monoids
by construction.  The fact that also
$
\xymatrix@C=15pt{
B\bullet I\ar[r]^-{B\bullet\tau}
& B\bullet J\ar[r]^-{\cong}
& B }
$
is a morphism of monoids, follows from the commutativity of
$$
\xymatrix@R=15pt @C=35pt{
(B\bullet I)\circ(B\bullet I)
\ar[rr]^-{(B\bullet\tau)\circ(B\bullet \tau)}\ar[d]_-{\zeta}
&& (B\bullet J)\circ(B\bullet J)\ar[r]^-{\cong}\ar[d]_-{\zeta}
&  B\circ B \ar[ddd]^-{\mu}\\
(B\circ B)\bullet (I\circ I)
\ar[rr]^-{(B\circ B)\bullet(\tau\circ\tau)}
\ar[d]_-{\cong}
&& (B\circ B)\bullet(J\circ J)\ar[d]_-{(B\circ B)\circ\varpi}  &  \\
(B\circ B)\bullet I \ar[d]_-{\mu \bullet I} \ar[rr]^-{(B\circ B)\bullet\tau}
&& (B\circ B)\bullet J\ar[d]_-{\mu\bullet J}\ar@/_2pc/[uur]^-{\cong} & \\
B\bullet I  \ar[rr]_-{B\bullet\tau}
&& B\bullet J \ar[r]_-{\cong}
&  B}
$$
--- where the middle square on the left commutes by unitality of the monoid
$(J,\varpi,\tau)$ and the region at the top right commutes by one of the
unitality axioms in \eqref{eq1.2} --- and 
$$
\xymatrix{
I   \ar@/^2pc/[ddrr]^-{\eta}\ar[rd]^-{\cong} \ar[d]_-{\delta} && \\
I\bullet I\ar[r]^-{I\bullet\tau} \ar[d]_-{\eta\bullet I}
& I\bullet J\ar[d]^-{\eta\bullet J} & \\
B\bullet I \ar[r]_-{B\bullet\tau} & B\bullet J \ar[r]_-{\cong}
& B}
$$
--- where the top left triangle commutes by the counitality of the comonoid
$(I,\delta,\tau)$. 
\end{proof}

The morphism $\omega$ in Proposition \ref{prop:omega} induces a left
$B^{\co}$-action
$\gamma:=\big(
\xymatrix@C=18pt{
B^{\co}\circ B\ar[r]^-{\omega\circ B}
&  B\circ B \ar[r]^-{\mu}
& B }\big)
$.
If the coequalizer
\begin{equation}\label{eq:BcoA_product}
\xymatrix@C=35pt{
(P\circ B^{\co})\circ B 
\ar@<2pt>[rr]^-{\gamma\circ B}
\ar@<-2pt>[rr]_-{(P\circ\gamma).\alpha}
&&  P\circ B \ar[r]^{\pi_{P,B}}
&  P\circ_{B^{\co}} B }
\end{equation}
in $\mathcal M$ exists --- for any right $B^{\co}$-module $(P,\gamma)$ 
--- and any power of $(-)\circ B:\mathcal M\to \mathcal M$ preserves 
the coequalizers (\ref{eq:BcoA_product}), then there is an adjunction
$$
\xymatrix@C=40pt{
\mathcal{M}_{B^{\co}} \ar@{->}@/^1pc/[rr]^{\omega_\ast=(-)\circ_{B^{\co}}B}
\ar@{}[rr]|-{\perp}
&& \mathcal{M}_B,\ar@{->}@/^1pc/[ll]^{\omega^\ast} }
$$
see \cite{Par}. The functor $\omega^*$ takes a right $B$-module $(Q,\gamma)$
to the $B^{\co}$-module $(Q,\gamma.(Q\circ \omega))$, and it acts on the
morphisms as the identity map. For any right $B^{\co}$-module $P$, the
$B$-action on $P\circ_{B^{\co}} B$ is constructed using the universality of
the coequalizer 
\begin{equation}
\xymatrix@C=35pt @R=15pt{
((P\circ B^{\co})\circ B)\circ B
\ar@<2pt>[rr]^-{(\gamma\circ B)\circ B}
\ar@<-2pt>[rr]_-{((P\circ\gamma).\alpha)\circ B}
&&(P\circ B)\circ B  \ar[d]_-\cong\ar[r]^-{\pi_{P,B}\circ B}
&  (P\circ_{B^{\co}}B)\circ B
\ar@{..>}@/1.5pc/[dd]^-\gamma\\
&&P\circ(B\circ B)\ar[d]_-{P\circ\mu}   \\
&&P\circ B\ar[r]_-{\pi_{P,B}}
& P\circ_{B^{\co}}B}
\label{eq2.7}
\end{equation}
Thus $\pi_{P,B}$ is a morphism of $B$-modules by construction.

By Proposition \ref{prop:BcoA_monoid}, $B^{\co}$ is a monoid in
$\mathcal{M}^I$ whenever it exists. We denote by ${\mathcal M}_{B^{\co}}^I$
the category of $B^{\co}$-modules in ${\mathcal M}^I$.

\begin{proposition} \label{prop:lambda}
For a right comodule monoid $B$ over a bimonoid $A$ in a duoidal
category $\mathcal M$, assume that the equalizer (\ref{eq:BcoA}) and the
coequalizers (\ref{eq:BcoA_product}) exist and that any power of
$(-)\circ B:\mathcal M\to \mathcal M$ preserves the coequalizers
(\ref{eq:BcoA_product}). Then the functor $(-)\circ_{B^{\co}} B: {\mathcal
M}_{B^{\co}}\to {\mathcal M}_B$ lifts to ${\mathcal M}_{B^{\co}}^I\to
{\mathcal M}_B^A$.
\end{proposition}

\begin{proof}
In light of \cite[Corollary 5.11]{PowWat} or \cite[Proposition 1.1 and Theorem
1.2]{GomTor} (see also Section \ref{sec:Dubuc-Beck}), we need to construct a
morphism $\lambda_P: (P\bullet I)\circ_{B^{\co}} B \to (P \circ_{B^{\co}}
B)\bullet A$ in ${\mathcal M}_B$, for any right $B^{\co}$-module $P$; and show
that $\lambda$ is in fact a comonad morphism. The morphism $\lambda_P$ is
constructed by using the universality of the coequalizer in $\mathcal{M}$ in
the top row:  
\begin{equation}
\xymatrix{
((P\bullet I)\circ B^{\co})\circ B
\ar@<2pt>[rr]^-{\gamma \circ B}
\ar@<-2pt>[rr]_-{((P\bullet I)\circ\gamma).\alpha}
&&  (P\bullet I)\circ B \ar[d]_-{\lambda_P^0}\ar[rr]^-{\pi_{P\bullet I,B}}
&&  (P\bullet I)\circ_{B^{\co}}B  \ar@{..>}@/1.5pc/[d]^-{\lambda_P}\\
&& (P\circ B)\bullet A\ar[rr]^-{\pi_{P,B}\bullet A}
&& (P\circ_{B^{\co}} B)\bullet A,}
\label{eq2.8}
\end{equation}
where $\lambda^0$ is the comonad morphism (\ref{eq:lambda^0}).
In order to see that $\lambda_P$ is well defined, we need to check that
$(\pi_{P,B}\bullet A).\lambda_P^0$ coequalizes the parallel morphisms in the
top row. 
The $B^{\co}$-action on $P\bullet I$ is given by 
$$
\xymatrix@C=12pt{
(P\bullet I)\circ B^{\co}\ar[rr]^-{(P \bullet I)\circ \rho}&&
(P\bullet I)\circ (B^{\co}\bullet I)\ar[r]^-\zeta&
(P\circ B^{\co})\bullet (I\circ I)\ar[r]^-\cong &
(P\circ B^{\co})\bullet I\ar[r]^-{\gamma\bullet I}&
P\bullet I.}
$$
With this expression at hand, it follows 
by the naturality of $\zeta$ and of the unit constraints, 
since $\iota$ is a morphism of $I$-comodules,
by the coequalizer property of $\pi_{P,B}$, and 
since $\delta:I\to I\bullet I$ is counital; 
that $(\pi_{P,B}\bullet A).\lambda_P^0.(\gamma\circ B)$ is equal to the image
of  
$$
\xymatrix{
(P\bullet I)\circ B^{\co} \ar[r]^-{(P\bullet I)\circ \iota}&
(P\bullet I)\circ (B\bullet I)\ar[r]^-\zeta&
(P\circ B)\bullet (I\circ I)\ar[r]^-\cong&
(P\circ B)\bullet I}
$$
under the functor $(-)\circ B$, composed with 
$$
\xymatrix{
((P\circ B)\bullet I)\circ B\ar[r]^-{\lambda^0_{P\circ B}}&
((P\circ B)\circ B)\bullet A\ar[r]^-{\gamma\bullet A}&
(P\circ B)\bullet A\ar[r]^-{\pi_{P,B}\bullet A}&
(P\circ_{B^{\co}} B)\bullet A.}
$$
On the other hand, $(\pi_{P,B}\bullet A).\lambda_P^0.((P\bullet
I)\circ\gamma).\alpha$ is equal to the same morphism. This follows by the
explicit form of the $B^{\co}$ action on $B$ induced by the morphism $\omega$
in Proposition \ref{prop:omega}, by the naturality of $\zeta$ and of the unit
constraints, by the first compatibility condition in \eqref{eq2.1}, by the
equalizer property of $\iota:B^{\co}\to B\bullet I$, by counitality of the
comonoid $(I,\delta,\tau)$, by one of the associativity axioms
in \eqref{eq1.1}, and by unitality of the monoid $(A,\mu,\eta)$. So by
universality, the morphism $\lambda_P$ exists. 

$\lambda$ is natural since $\lambda^0$ and $\pi$ are so.
It was proven in Proposition \ref{prop:rel_Hopf_lifting} that
$\lambda_P^0$ is a morphism of $B$-modules. Then so is $(\pi_{P,B}\bullet
A).\lambda_P^0$ hence also $\lambda_P$.
The compatibilities of $\lambda$ with the comultiplications and the
counits of the comonads $(-)\bullet I$ on ${\mathcal M}_{B^{\co}}$ and
$(-)\bullet A$ on ${\mathcal M}_B$ follow by naturality of $\pi$ in its first
argument and the compatibilities of $\lambda^0$ in
Proposition \ref{prop:rel_Hopf_lifting}. 
\end{proof}

Whenever the assumptions in Proposition \ref{prop:lambda} hold, it follows by
Proposition \ref{prop:lambda} that for any object $N$ in ${\mathcal
M}_{B^{\co}}^I$, $N\circ_{B^{\co}} B$ is a right $A$-comodule via the coaction
in the bottom row of the following diagram. What is more, $\pi_{N,B}$ is also
a morphism of $A$-comodules by commutativity of the diagram.
$$
\xymatrix@R=15pt{
(N\bullet I)\circ(B\bullet A) \ar[rr]^-{\zeta}
&& (N\circ B)\bullet(I\circ A)\ar[d]^-{\cong}\\
N\circ B\ar[u]^{\rho\circ\rho}\ar[r]^-{\rho\circ B}\ar[d]_-{\pi_{N,B}}
& (N\bullet I)\circ B \ar[d]^-{\pi_{N\bullet I,B}}\ar[r]^-{\lambda^0_N}
& (N\circ B)\bullet A\ar[d]^-{\pi_{N,B}\bullet A}\\
N\circ_{B^{\co}}B\ar[r]_-{\rho\circ_{B^{\co}}B}
& (N\bullet I)\circ_{B^{\co}}B\ar[r]_-{\lambda_N} &
(N\circ_{B^{\co}} B)\bullet A. }
$$
Commutativity of the region at the top follows immediately from the form
of $\lambda^0_N$, see (\ref{eq:lambda^0}).

We have seen that the functor
$(-)^{\co}: \mathcal{M}_B^A\rightarrow \mathcal{M}^I$ --- whenever it
exists --- is right adjoint of the functor $(-)\circ
B: \mathcal{M}^I\rightarrow \mathcal{M}^A_B$. Our next aim is to show
that it factorizes through the right adjoint of $(-)\circ_{B^{\co}}
B:{\mathcal M}_{B^{\co}}^I \to {\mathcal M}_B^A$, via the forgetful functor
${\mathcal M}_{B^{\co}}^I \to {\mathcal M}^I$.

\begin{proposition} \label{prop:adjoints}
For a right comodule monoid $B$ over a bimonoid $A$ in a duoidal
category $\mathcal M$, assume that the equalizers (\ref{eq:McoA}) and the
coequalizers (\ref{eq:BcoA_product}) exist and that any power of
$(-)\circ B:\mathcal M\to \mathcal M$ preserves the coequalizers
(\ref{eq:BcoA_product}). Then there is a lifting of
$(-)^{\co}: \mathcal{M}_B^A\rightarrow \mathcal{M}^I$ to
$\mathcal{M}_B^A\rightarrow \mathcal{M}^I_{B^{\co}}$, which provides the right
adjoint of the functor $(-)\circ_{B^{\co}} B$ in
Proposition \ref{prop:lambda}. 
\end{proposition}

\begin{proof} 
We are to prove that there is an adjoint triangle
\begin{equation}\label{eq:bigtriangle}
\xymatrix@C=15pt{
&&&&&&  \mathcal{M}^A_B \ar@{..>}@/^1pc/[llllllddd]^-{N=(-)^{\co}}
\ar@{->}@/_0.5pc/[ddd]_{U^A}^-\dashv\\
&&&&&& \\   &&&&&& \\
\mathcal{M}^I_{B^{\co}}\ar@{->}@/^0.5pc/[rrr]^(.65){U^I}_-\perp
\ar@{->}@/^1pc/[rrrrrruuu]^{K=(-)
\raisebox{-5pt}{$\stackrel\circ{{}_{B^{\co}}}$} B}_-\perp
&&&  \mathcal{M}_{B^{\co}}
\ar@{->}@/^0.5pc/[rrr]^{\omega_\ast=(-)
\raisebox{-5pt}{$\stackrel\circ{{}_{B^{\co}}}$} B}_-\perp
\ar@{->}@/^0.5pc/[lll]^{F^I=(-)\bullet I}
&&&  \mathcal{M}_B, \ar@{->}@/^0.5pc/[lll]^{\omega^\ast}
\ar@{->}@/_0.5pc/[uuu]_{F^A=(-)\bullet A}}
\end{equation}
in which $U^AK=\omega_\ast U^I$.
By the Adjoint Lifting Theorem (cf. the dual form of \cite[Theorem
2.1]{J. Power}, see Section \ref{sec:Dubuc-Beck}), the functor $N$ exists, and
it is a right adjoint of $K$, provided that the equalizer  
\begin{equation}\label{eq:eq_MBcoA}
\xymatrix{
X^{\co}\ar[r]&
X\bullet I\ar@<2pt>[r]\ar@<-2pt>[r]&
(X\bullet A)\bullet I}
\end{equation}
exists in $\mathcal M^I_{B^{\co}}$, for any $(A,B)$-relative Hopf module $X$.
The upper one of the parallel arrows is $\varphi^0=\rho\bullet I$ as
in \eqref{eq:McoA}. The lower one is 
\begin{equation}\label{eq:phi1}
\xymatrix@C=15pt @R=8pt{
X\bullet I\ar[r]^-{\nu^I}&
(X\bullet I)\bullet I \ar[r]^-{\nu\bullet I}&
((X\bullet I)\raisebox{-5pt}{$\stackrel\circ{{}_{B^{\co}}}$} B )\bullet I
\ar[r]^-{\rho \bullet I}&
(((X\bullet I)\raisebox{-5pt}{$\stackrel\circ{{}_{B^{\co}}}$} B )\bullet A)
\bullet I
\ar[rr]^-{((\epsilon^I\circ_{B^{\co}}B)  \bullet A) \bullet I}&&{\ } \\
&&&
((X\raisebox{-5pt}{$\stackrel\circ{{}_{B^{\co}}}$} B)\bullet A)\bullet I
\ar[rr]^-{(\epsilon\bullet A)\bullet I}&&
(X\bullet A)\bullet I.}
\end{equation}
Here 
$\nu^I$ is the unit, and $\epsilon^I$ is the counit of the adjunction
$U^I\dashv F^I$,
and $\nu$ is the unit and $\epsilon$ is the counit of the adjunction
$\omega_\ast \dashv \omega^\ast$. 
(Recall that for any right $B$-module $Q$, $\epsilon_Q.\pi_{Q,B}$ is
equal to the $B$-action on $Q$.)
The symbol $\rho$ denotes the $A$-coaction
on $(X\bullet I)\circ_{B^{\co}} B$ from Proposition \ref{prop:lambda}.
A computation --- using that $\pi_{X\bullet I,B}$ and the unit $\eta:I\to B$
are morphisms of $A$-comodules, 
naturality of $\pi$ in the first argument,
the relation between the counit $\epsilon_X$ and the $B$-action on $X$,
counitality of the comonoid $I$ and unitality of the $B$-action on $X$,
and one of the unitality axioms in \eqref{eq1.2} --- 
yields that \eqref{eq:phi1} is equal to $\varphi^1=((X\bullet \eta)\bullet
I).\alpha^{-1}.(X\bullet \delta)$ as in \eqref{eq:McoA}. 

The $B^{\co}$-action on $X\bullet I$ is equal to 
$$
\xymatrix@C=10pt @R=8pt{
(X\bullet I)\circ B^{\co}\ar[rr]^-{(X\bullet I)\circ \rho}&&
(X\bullet I)\circ (B^{\co}\bullet I)\ar[r]^-{\zeta}&
(X\circ B^{\co})\bullet (I\circ I)\ar[r]^-\cong&
(X\circ B^{\co})\bullet I\ar[r]^-{(X\circ \omega)\bullet I}&\\
&&&&
(X\circ B)\bullet I \ar[r]^-{\gamma\bullet I}&
X\bullet I.}
$$
Using the explicit form of $\omega$ in Proposition \ref{prop:omega}, the fact
that $\iota:B^{\co}\to B\bullet I$ is a morphism of $I$-comodules and
counitality of the comonoid $I$, this $B^{\co}$-action is shown to be equal to 
$$
\xymatrix{
(X\bullet I)\circ B^{\co}\ar[r]^-{(X\bullet I)\circ \iota}&
(X\bullet I)\circ (B\bullet I)\ar[r]^-\zeta&
(X\circ B)\bullet (I\circ I)\ar[r]^-\cong &
(X\circ B)\bullet I \ar[r]^-{\gamma\bullet I}&
X\bullet I.}
$$
The $B^{\co}$-action on $X\bullet A$ is equal to 
$$
\xymatrix@C=12pt{
(X\bullet A)\circ B^{\co}\ar[rr]^-{(X\bullet A)\circ \omega}&&
(X\bullet A)\circ B\ar[rr]^-{(X\bullet A)\circ \rho}&&
(X\bullet A)\circ (B\bullet A)\ar[r]^-\zeta&
(X\circ B)\bullet (A\circ A)\ar[r]^-{\gamma\bullet \mu}&
X\bullet A.}
$$
Using the equalizer property of $\iota:B^{\co}\to B\bullet I$, counitality of
the comonoid $I$ and unitality of the monoid $A$, this $B^{\co}$-action turns
out to be equal to 
$$
\xymatrix@C=12pt{
(X\bullet A)\circ B^{\co}\ar[rr]^-{(X\bullet A)\circ \iota}&&
(X\bullet A)\circ (B\bullet I)\ar[r]^-\zeta&
(X\circ B)\bullet (A\circ I)\ar[r]^-\cong &
(X\circ B)\bullet A \ar[r]^-{\gamma\bullet A}&
X\bullet A.}
$$
With these $B^{\co}$-actions at hand, $X\bullet \eta:X\bullet I\to X\bullet A$
is a morphism of $B^{\co}$-modules by naturality of $\zeta$ and
functoriality of both monoidal structures. 
So since $\varphi^0=\rho\bullet I$ and $(X\bullet \eta)\bullet I$ are in the
range of the functor $(-)\bullet I:\mathcal M_{B^{\co}}\to \mathcal
M_{B^{\co}}$, they are morphisms of $B^{\co}$-modules. 
Since $\alpha^{-1}.(X\bullet \delta)$ is the comultiplication of the comonad
$(-)\bullet I:\mathcal M_{B^{\co}}\to \mathcal M_{B^{\co}}$ evaluated at $X$,
it is a morphism of $B^{\co}$-modules too. Thus both $\varphi^0$ and
$\varphi^1$ are morphisms of $B^{\co}$-modules. Since the forgetful functor
$\mathcal M^I_{B^{\co}}\to \mathcal M^I$ creates equalizers,
and \eqref{eq:eq_MBcoA} is an equalizer in $\mathcal M^I$ by assumption, we
conclude that it is an equalizer in $M^I_{B^{\co}}$ as needed.
\end{proof}

\subsection{Galois extensions}

For a right comodule monoid $B$ over a bimonoid $A$ in a duoidal
category $\mathcal M$, assume that the equalizers (\ref{eq:McoA}) and the
coequalizers (\ref{eq:BcoA_product}) exist and that any power of
$(-)\circ B:\mathcal M\to \mathcal M$ preserves the coequalizers
(\ref{eq:BcoA_product}). Then by Proposition \ref{prop:adjoints} there is an
adjoint triangle \eqref{eq:bigtriangle}. Hence as in \eqref{eq:beta_long},
there is a corresponding morphism of comonads 
$\beta$:  
$$
\xymatrix@C=45pt@R=5pt{
\omega_\ast U^IF^I\omega^\ast\ar@{=}[d]
&U^AF^A\omega_\ast U^IF^I\omega^\ast\ar@{=}[d]
\ar[r]^-{U^AF^A\omega_\ast \epsilon^I \omega^\ast}
& U^AF^A\omega_\ast \omega^\ast\ar[r]^-{U^AF^A\epsilon}
& U^AF^A,\\
U^AKF^I\omega^\ast\ar[r]^-{U^A\nu^AKF^I\omega^\ast}
&U^AF^AU^AKF^I\omega^\ast}
$$
where $\nu^A$ is the unit of the adjunction $U^A\dashv F^A$, and $\epsilon^I$
and $\epsilon$ are the counits of the adjunctions $U^I\dashv F^I$ and
$\omega_\ast \dashv \omega^\ast$, respectively.
Explicitly, for any right $B$-module $Q$, $\beta_Q$ is the
morphism 
\begin{equation}\label{eq:beta_M}
\xymatrix @C=15pt @R5pt{
(Q\bullet I)
\raisebox{-5pt}{$\stackrel\circ{{}_{B^{\co}}}$} B \ar[r]^-{\rho}
& ((Q\bullet I)
\raisebox{-5pt}{$\stackrel\circ{{}_{B^{\co}}}$} B)\bullet A
\ar[rrr]^{((Q\bullet \tau)\circ_{B^{\co}}B)\bullet A}
&&& ((Q\bullet J)
\raisebox{-5pt}{$\stackrel\circ{{}_{B^{\co}}}$} B)\bullet A \ar[r]^-{\cong}
&\\
&&&& (Q
\raisebox{-5pt}{$\stackrel\circ{{}_{B^{\co}}}$} B)\bullet A
\ar[r]^-{\epsilon_Q\bullet A}
& Q\bullet A. }
\end{equation}

\begin{proposition}
For a right comodule monoid $B$ over a bimonoid $A$ in a duoidal
category $\mathcal M$, assume that the equalizers (\ref{eq:McoA}) and  the
coequalizers (\ref{eq:BcoA_product}) exist and that any power of
$(-)\circ B:\mathcal M\to \mathcal M$ preserves the coequalizers
(\ref{eq:BcoA_product}).
For any right $B$-module $(Q,\gamma)$, consider the natural
transformation 
\begin{equation}
\label{eq:beta0M}
\beta^0_Q:=\big(
\xymatrix@C=10pt{
(Q\bullet I)\circ B \ar[rr]^-{(Q\bullet I)\circ \rho}&&
(Q\bullet I)\circ (B\bullet A)\ar[r]^-\zeta&
(Q\circ B)\bullet (I\circ A)\ar[r]^-{\cong}&
(Q\circ B)\bullet A \ar[r]^-{\gamma\bullet A}&
Q\bullet A } \big).
\end{equation}
Then $\beta_Q$ in (\ref{eq:beta_M}) can be characterized as the unique
morphism for which $\beta^0_Q=\beta_Q.\pi_{Q\bullet I,B}$.
\end{proposition}

\begin{proof}
By naturality of $\pi$ and by $\epsilon_Q.\pi_{Q,B}$ being equal to the
$B$-action on $Q$, $\beta_Q$ is the unique morphism for which the diagram 
$$
\xymatrix @C=15pt{
(Q\bullet I)\raisebox{-5pt}{$\stackrel\circ{{}_{B^{\co}}}$}B
\ar[rrrrr]^-{\beta_Q}
&&&&& Q\bullet A \\
(Q\bullet I)\circ B \ar[r]_-{\rho}
\ar[u]^-{\pi_{Q\bullet I,B}}
& ((Q\bullet I)\circ B)\bullet A \ar[rr]_-{((Q\bullet\tau)\circ B)\bullet A}
&&((Q\bullet J)\circ B)\bullet A \ar[r]_-\cong
& (Q\circ B)\bullet A \ar[r]_-{\gamma\bullet A}
&Q\bullet A \ar@{=}@/^0pc/[u]}
$$
commutes. (The $A$-coaction $\rho$ on $(Q\bullet I)\circ B$ appearing on the
left in the bottom row, is induced from the $I$-coaction on $Q\bullet I$ by
the comonad morphism $\lambda^0$ in (\ref{eq:lambda^0}).)
The morphism in the bottom row is equal to (\ref{eq:beta0M})
by the explicit form of the $A$-coaction on $(Q\bullet I)\circ B$, by the
counitality of the comonoid $(I,\delta,\tau)$ and by naturality of $\zeta$
and of the unit constrains.
\end{proof}

\begin{definition} Consider a duoidal category $\mathcal M$. A monoid
morphism $\omega:C\rightarrow B$ in $(\mathcal M,\circ,I)$ is called a {\it
Galois extension} by a bimonoid $A$ in $\mathcal M$ if the following
conditions hold. 
\begin{itemize}
\item $B$ is a right comodule monoid over $A$.
\item The $A$-coinvariant part of any $(A,B)$-relative Hopf module (i.e. the
  equalizer (\ref{eq:McoA})) exists.
\item $C$ fits the equalizer diagram
$$
\xymatrix{
C\ar[r]^-\iota&
B\bullet I \ar@<2pt>[r]^-{\varphi_0}\ar@<-2pt>[r]_-{\varphi_1}&
(B\bullet A)\bullet I}
$$
cf. \eqref{eq:BcoA} (so that $C$ is the coinvariant part of $B$), and
$\omega:C\to B$ is the corresponding morphism of monoids in
Proposition \ref{prop:omega}. 
\item The coequalizers (\ref{eq:BcoA_product}) exist and any power of
$(-)\circ B:\mathcal M\to \mathcal M$ preserves them.
\item The natural transformation $\beta$ in (\ref{eq:beta_M}) is an
isomorphism.
\end{itemize}
\end{definition}

If $\mathcal M$ is a braided monoidal category, this reduces to the
usual definition of a Galois extension by a bimonoid (see
e.g. \cite[Definition 3.1]{Scha}). 

\section{The fundamental theorem of Hopf modules} \label{sec:fthm}

In this section we analyze further the particular extension $I\to A$ by
a bimonoid $A$ in a duoidal category $\mathcal M$. Making some assumptions on
$\mathcal M$, we relate its Galois property to the Fundamental Theorem of Hopf
Modules holding true --- that is, to an equivalence between the category of
$A$-Hopf modules and the category ${\mathcal M}^I$ of comodules over the
$\circ$-monoidal unit $I$. (Clearly, if $\mathcal M$ is a braided monoidal
category, ${\mathcal M}^I$ is isomorphic to $\mathcal M$.)

\subsection{Properties of the Galois morphism}

Let $A$ be a bimonoid in a duoidal category $\mathcal M$ and regard $A$
as a right $A$-comodule monoid (via the coaction provided by the
comultiplication). The resulting category of $A$-modules in the category
${\mathcal M}^A$ of $A$-comodules is denoted by ${\mathcal M}^A_A$ and it is
called {\em the category of $A$-Hopf modules}.

\begin{lemma} \label{lem:AcoA}
Let $A$ be a bimonoid in a duoidal category $\mathcal M$. Then the coinvariant
part of $A$ as a Hopf module exists and it is isomorphic to the
$\circ$-monoidal unit $I$. In fact, the following is a contractible equalizer
in $\mathcal M^I$ 
\begin {equation}
\xymatrix{
I\ar[r]^-{\delta}
&  I\bullet I\ar[r]^-{\eta\bullet I}
&  A\bullet I \ar@<2pt>[rrr]^-{\varphi^0=\Delta\bullet I}
\ar@<-2pt>[rrr]_-{\varphi^1=((A\bullet \eta)\bullet I).\alpha^{-1}.
(A\bullet \delta)}
&&&  (A\bullet A)\bullet I.}
\label{eq3.2}
\end{equation}
\end{lemma}

\begin{proof}
The diagram in (\ref{eq3.2}) is a fork by the the coassociativity
of $\delta$ and the third axiom in Definition \ref{Def:bimonoid}.
The contracting morphisms are 
$$
\xymatrix@C=15pt{
(A\bullet A)\bullet I\ar[rr]^-{(\varepsilon\bullet A)\bullet I}
&& (J\bullet A)\bullet I\ar[r]^-\cong
& A\bullet I}
\quad \textrm{and}\quad
\xymatrix@C=15pt{
A\bullet I\ar[r]^-{\varepsilon\bullet I}
& J\bullet I\ar[r]^-\cong
& I.}
$$
This is seen applying the fourth bimonoid axiom in
Definition \ref{Def:bimonoid}, counitality of the comonoids $I$ and $A$, 
functoriality of both monoidal products and coherence.
\end{proof}

The most important consequence of Lemma \ref{lem:AcoA} is that the
relative product $\circ_{A^{\co}}$ reduces to the monoidal product
$\circ_I\equiv \circ$ in $\mathcal M$. In particular, the
corresponding coequalizer (\ref{eq:BcoA_product}) is trivial. Hence it exists
and it is preserved by any functor. 

Lemma \ref{lem:AcoA} also implies that in the case of the $A$-comodule monoid
$A$, the difference between $\beta_Q$ and $\beta_Q^0$ in Proposition 2.8
disappears, they become equal. Moreover, substituting $M'=I$ in (\ref{eq1.6}),
and $B=A$ in (\ref{eq:beta0M}), the resulting morphisms $\beta_{Q,I}$ and
$\beta_Q$ differ by an isomorphism (a unit constraint in $\mathcal M$). This
means that
\begin{equation}\label{eq:beta^A}
\beta_Q=\big(
\xymatrix@C=12pt{
(Q\bullet I)\circ A \ar[rr]^-{(Q\bullet I)\circ \Delta}&&
(Q\bullet I)\circ (A\bullet A)\ar[r]^-\zeta&
(Q\circ A)\bullet (I\circ A)\ar[r]^-{\cong}&
(Q\circ A)\bullet A \ar[r]^-{\gamma\bullet A}&
Q\bullet A
} \big)
\end{equation}
is an isomorphism for any right $A$-module $(Q,\gamma)$; that is, $\beta$
is a natural isomorphism, whenever $(-)\circ A$ is a right Hopf monad on
$\mathcal M$. (However, the converse implication needs not be true.) In the
rest of this section we study its properties. 

\begin{lemma} \label{Lemma1}
Let $A$ be a bimonoid in a duoidal category $\mathcal M$. Then the
natural transformation $\beta$ in (\ref{eq:beta^A}) obeys the following
compatibility with the counit, for any right $A$-module $(Q,\gamma)$.
$$
\xymatrix@C=30pt{
(Q\bullet I)\circ A \ar[r]^-{(Q\bullet\tau)\circ A}
\ar[d]_-{\beta_{Q}}
&(Q\bullet J)\circ A\ar[r]^-{\cong}
&Q\circ A\ar[d]^-{\gamma}\\
Q\bullet A\ar[r]_-{Q\bullet \varepsilon}
& Q\bullet J\ar[r]_-{\cong}
& Q }
$$
\end{lemma}

\begin{proof}
The claim is verified using the counitality of the comonoid $A$
and \eqref{eq:M68}. 
\end{proof}

\begin{lemma} \label{Lemma2}
Let $A$ be a bimonoid in a duoidal category $\mathcal M$. Then the
natural transformation $\beta$ in (\ref{eq:beta^A}) obeys the following
compatibility with the unit, for any right $A$-module $(Q,\gamma)$.
$$
\xymatrix@C=30pt{
Q\bullet I \ar[r]^-{\cong} \ar[rrd]_-{Q\bullet\eta}
&(Q\bullet I)\circ I
\ar[r]^-{(Q\bullet I)\circ\eta}
&(Q\bullet I)\circ A \ar[d]^-{\beta_{Q}}\\
&& Q\bullet A}
$$
\end{lemma}

\begin{proof}
This claim follows by the third axiom of a bimonoid in
Definition \ref{Def:bimonoid}, unitality of the $A$-action 
on $Q$ and one of the unitality axioms in \eqref{eq1.2}.
\end{proof}

\begin{lemma} \label{Lemma3}
Let $A$ be a bimonoid in a duoidal category $\mathcal M$. Then the
natural transformation $\beta$ in (\ref{eq:beta^A}) obeys the following
compatibility with the unit and the comultiplication, for any right $A$-module
$(Q,\gamma)$.
$$
\xymatrix{
(Q\bullet I)\circ A
\ar[r]^-{(Q\bullet\delta)\circ A}\ar[d]_-{\beta_{Q}}
&(Q\bullet(I\bullet I))\circ A\ar[r]^-{\cong}
&((Q\bullet I)\bullet I)\circ A\ar[rr]^-{((Q\bullet \eta)\bullet I)\circ A}
&&((Q\bullet A)\bullet I)\circ A\ar[d]^-{\beta_{Q\bullet A}} \\
Q\bullet A\ar[rr]_-{Q\bullet \Delta}
&&Q\bullet(A\bullet A)\ar[rr]_-{\cong}
&& (Q\bullet A)\bullet A }
$$
\end{lemma}

\begin{proof}
Coassociativity of the comonoid $A$, one of the associativity axioms
in \eqref{eq1.1} and one of the unitality axioms in (\ref{eq1.2}) imply
the claim. 
\end{proof}

\begin{lemma} \label{Lemma4}
Let $A$ be a bimonoid in a duoidal category $\mathcal M$. Then the
natural transformation $\beta$ in (\ref{eq:beta^A}) obeys the following
compatibility with the unit and the comultiplication, for any object $M$ in
$\mathcal M$.
$$
\xymatrix{
(M\bullet I)\circ A \ar[d]_-{(M\bullet I)\circ\Delta}\ar[r]^-{\cong}
&((M\circ I)\bullet I)\circ A \ar[rr]^-{((M\circ \eta)\bullet I)\circ A}
&&((M\circ A)\bullet I)\circ A \ar[d]^-{\beta_{M\circ A}} \\
(M\bullet I)\circ (A\bullet A)\ar[r]_-{\zeta }
& (M\circ A)\bullet (I\circ A)\ar[rr]_-{\cong}
&& (M\circ A)\bullet A }
$$
\end{lemma}

\begin{proof}
The claim is obtained from the unitality of the monoid $A$ and from coherence
and naturality of $\zeta$ and of the associativity and unit constraints. 
\end{proof}

\subsection{The existence of coinvariants}\label{sec:Hopf_coinv}

The aim of this section is to prove --- under certain assumptions
on a duoidal category $\mathcal M$ --- that the coinvariant part of any Hopf
module over a bimonoid $A$ in $\mathcal M$ exists whenever (\ref{eq:beta^A})
is a natural isomorphism.

\begin{proposition} \label{prop:split}
Let $A$ be bimonoid in a duoidal category $\mathcal M$
such that (\ref{eq:beta^A}) is a natural isomorphism. Then for
any $A$-Hopf module $(X,\gamma:X\circ A \to X, \rho:X\to X\bullet A)$, the
parallel pair of morphisms
$$
\xymatrix@C=50pt{
X\bullet I \ar@<2pt>[rr]^-{\varphi^0=\rho\bullet I}
\ar@<-2pt>[rr]_-{\varphi^1=
((X\bullet \eta)\bullet I).\alpha^{-1}.(X\bullet \delta)}
&&  (X\bullet A)\bullet I}
$$
is contractible by the functor 
$
H:=\big(
\xymatrix@C=10pt{
\mathcal{M}^I \ar[r]^-{U^I}
& \mathcal{M}\ar[r]^-{(-)\circ J}
& \mathcal{M}_J}\big).
$
\end{proposition}

\begin{proof}
Introduce the following morphism to be called $\theta$: 
$$
\xymatrix@C=10pt@R=8pt{
((X\bullet A)\bullet I)\circ J 
\ar[rrr]^-{(\beta_X^{-1}\bullet I)\circ J}
&&& (((X\bullet I)\circ A)\bullet I)\circ J
 \ar[rrrr]^-{(((X\bullet I)\circ \varepsilon)\bullet \tau)\circ J}
&&&& (((X\bullet I)\circ J)\bullet J)\circ J\ar[r]^-\cong&\\
&&&  \qquad\ (X\bullet I)\circ(J\circ J)\ar[rrrr]^-{(X\bullet I)\circ\varpi}
&&&& (X\bullet I)\circ J.\ \ \qquad\qquad}
$$
It is natural in $X$ by the naturality of $\beta$, it is a
morphism of $J$-modules by the associativity of the multiplication
$\varpi:J\circ J\to J$, and it obeys $\theta.(\varphi^1\circ J)=(X\bullet
I)\circ J$ by Lemma \ref{Lemma2},
the fourth bimonoid axiom in Definition \ref{Def:bimonoid}, 
and the counitality of the comonoid $I$ and unitality of the monoid
$J$. 
In order to prove the equality $(\varphi^1\circ
J).\theta.(\varphi^0\circ J)=(\varphi^0\circ J).\theta.(\varphi^0\circ
J)$, note that in the diagram
$$
\xymatrix{
(X\bullet I)\circ J
\ar[rr]^-{\varphi^0\circ J=(\rho\bullet I)\circ J}
\ar[d]_-{(X\bullet\tau)\circ J}
&& ((X\bullet A)\bullet I)\circ J\ar[d]_-{((X\bullet A)\bullet\tau)\circ J}
\ar@/^7pc/[dddddd]^-\theta\\
(X\bullet J)\circ J \ar[d]_-{\cong} && ((X\bullet A)\bullet J)\circ J
\ar[d]_-{\cong}\\
X\circ J\ar[rr]^-{\rho\circ J} && (X\bullet A)\circ J
\ar[d]_-{\beta^{-1}_{X}\circ J} \\
(((X\bullet A)\bullet I)\circ A)\circ J
\ar[d]_{\cong}
&& ((X\bullet I)\circ A)\circ J
\ar[ll]_-{(\varphi^{0,1}\circ A)\circ J}\ar[d]_-{\cong} \\
((X\bullet A)\bullet I)\circ (A\circ J)
\ar[d]_-{((X\bullet A)\bullet I)\circ (\varepsilon\circ J)}
&&   (X\bullet I)\circ(A\circ J)
\ar[d]_-{(X\bullet I)\circ(\varepsilon\circ J)} \\
((X\bullet A)\bullet I)\circ (J\circ J)
\ar[d]_{((X\bullet A)\bullet I)\circ \varpi}
&&   (X\bullet I)\circ(J\circ J)\ar[d]_-{(X\bullet I)\circ\varpi} \\
((X\bullet A)\bullet I)\circ J
&& (X\bullet I)\circ J\ar[ll]^-{\varphi^{0,1}\circ J},}
$$
both regions commute by naturality, for either choice $\varphi^{0}$ or
$\varphi^{1}$ as $\varphi^{0,1}$.
So it suffices to check that
\begin{equation}\label{eq:epsinv}
\xymatrix{
X \ar[r]^-{\rho}
& X\bullet A\ar[r]^-{\beta^{-1}_{X}}
&(X\bullet I)\circ A
\ar@<2pt>[rrr]^-{\varphi^0\circ A}
\ar@<-2pt>[rrr]_-{\varphi^1\circ A}
&&&  ((X\bullet A)\bullet I)\circ A }
\end{equation}
is a fork. This follows 
by Lemma \ref{Lemma3},
by coassociativity of the $A$-coaction on $X$, 
and naturality of $\beta$ together with the fact that the $A$-coaction
on $X$ is a morphism of $A$-modules.
\end{proof}

\begin{corollary}\label{cor:split}
Let $\mathcal M$ be a duoidal category in which idempotent morphisms split and
equalizers of $
H:=\big(
\xymatrix{
\mathcal{M}^I \ar[r]^-{U^I}
& \mathcal{M}\ar[r]^-{(-)\circ J}
& \mathcal{M}_J}\big)
$-contractible equalizer pairs exist.
Then for any bimonoid $A$ in $\mathcal M$ such that (\ref{eq:beta^A})
is a natural isomorphism, there exists the equalizer
\begin{equation}\label{eq:XcoA}
\xymatrix{ X^{\co} \ar[r]^\iota
&X\bullet I \ar@<2pt>[rr]^-{\varphi^0} \ar@<-2pt>[rr]_-{\varphi^1}
&&  (X\bullet A)\bullet I.}
\end{equation}
It provides a right adjoint $(-)^{\co}$ of $(-)\circ A:{\mathcal M}^I\to
{\mathcal M}^A_A$.
\end{corollary}

\begin{proof}
By Proposition \ref{prop:split}, $(H\varphi^0,H\varphi^1)$ is a
contractible pair in $\mathcal M_J$. Since idempotent morphisms
are assumed to split in $\mathcal M$ --- hence also in $\mathcal M_J$
---, their (contractible) equalizer exists. That is to say,
$(\varphi^0,\varphi^1)$ is an $H$-contractible equalizer pair, so their
equalizer $X^{\co}$ exists by assumption. The final claim follows by the
considerations in Section \ref{sec:coinv}, see the text around
(\ref{eq:McoA}). 
\end{proof}

\subsection{Fully faithfulness}

Let $\mathcal M$ be a duoidal category in which idempotent morphisms
split and equalizers of 
$
H:=\big(
\xymatrix{
\mathcal{M}^I \ar[r]^-{U^I}
& \mathcal{M}\ar[r]^-{(-)\circ J}
& \mathcal{M}_J}\big)
$-contractible equalizer pairs exist.
Then combining the results in Section \ref{sec:rel_Hopf} and
Section \ref{sec:Hopf_coinv}, for any bimonoid $A$ in $\mathcal M$ such
that the canonical comonad morphism (\ref{eq:beta^A}) is an isomorphism, we
obtain an adjoint triangle 
\begin{equation}\label{eq:triangle}
\xymatrix{
&&&&&&  \mathcal{M}^A_A
\ar@{->}@/^1pc/[llllllddd]^{N =(-)^{\co}}
\ar@{->}@/_0.5pc/[ddd]_{U^A}^-\dashv\\
&&&&&& \\   &&&&&& \\
\mathcal{M}^I \ar@{->}@/^0.5pc/[rrr]^(.7){U^I}_-\perp
\ar@{->}@/^1pc/[rrrrrruuu]^{K=(-)\circ A}_-\perp
&&&  \mathcal{M} \ar@{->}@/^0.5pc/[rrr]^{F_A=(-)\circ A}_-\perp
\ar@{->}@/^0.5pc/[lll]^{F^I=(-)\bullet I}
&&&  \mathcal{M}_A. \ar@{->}@/^0.5pc/[lll]^{U_A}
\ar@{->}@/_0.5pc/[uuu]_{F^A=(-)\bullet A}}\
\end{equation}
The aim of this section is to find sufficient conditions for $K$ to be fully
faithful.

\begin{lemma}\label{lem:ZA_coinv}
Let $A$ be a bimonoid in a duoidal category $\mathcal M$ such that
(\ref{eq:beta^A}) is a natural isomorphism. Then for any right
$I$-comodule $Z$, $$
\xymatrix@C=15pt{
Z\circ J \ar[r]^-{\rho\circ J}
& (Z\bullet I)\circ J \ar[r]^-{\cong}
& ((Z\circ I)\bullet I)\circ J\ar[rr]^-{((Z\circ \eta)\bullet I)\circ J}
&& ((Z\circ A)\bullet I)\circ J
\ar@<2pt>[r]^-{\varphi^0\circ J}
\ar@<-2pt>[r]_-{\varphi^1\circ J}
&  (((Z\circ A)\bullet A)\bullet I)\circ J}
$$
is a contractible equalizer in ${\mathcal M}_J$.
\end{lemma}

\begin{proof}
We know from Proposition \ref{prop:split} that for the Hopf module
$Z\circ A$, $\theta:(((Z\circ A)\bullet A)\bullet I)\circ J \to ((Z\circ
A)\bullet I)\circ J$
obeys the equalities $\theta.(\varphi^1\circ J)=((Z\circ A)\bullet I)\circ J$
and $(\varphi^1\circ J).\theta.(\varphi^0\circ J)=(\varphi^0\circ
J).\theta.(\varphi^0 \circ J)$.
It follows 
by the fourth bimonoid axiom in Definition \ref{Def:bimonoid},
by the counitality of the $I$-coaction on $Z$ and unitality of the
monoid $J$, 
that
$$
\pi:=\big(
\xymatrix{
((Z\circ A)\bullet I)\circ J\ar[rr]^-{((Z\circ \varepsilon)\bullet \tau)\circ J}
&& ((Z\circ J)\bullet J)\circ J\ar[r]^-{\cong}
& Z\circ (J\circ J)\ar[r]^-{Z\circ\varpi}
& Z\circ J}\big)
$$
is a retraction of
$$
\upsilon:=
\big(
\xymatrix{
Z\circ J\ar[r]^-{\rho\circ J}&
(Z\bullet I)\circ J\ar[r]^-\cong&
((Z\circ I)\bullet I)\circ J\ar[rr]^-{((Z\circ \eta)\bullet I)\circ J}&&
((Z\circ A)\bullet I)\circ J}\big).
$$
Finally, $\upsilon.\pi=\theta.(\varphi^0\circ J)$
by  Lemma \ref{Lemma4}, and naturality of the associativity and unit
constraints.
\end{proof}

\begin{theorem}\label{thm:ff}
Let $\mathcal M$ be a duoidal category in which idempotent morphisms split and
$
H:=\big(
\xymatrix{
\mathcal{M}^I \ar[r]^-{U^I}
& \mathcal{M}\ar[r]^-{(-)\circ J}
& \mathcal{M}_J}\big)
$
is comonadic. Let $A$ be a bimonoid in $\mathcal M$. If (\ref{eq:beta^A}) is a
natural isomorphism, then the functor $K$ in (\ref{eq:triangle}) is fully
faithful.  
\end{theorem}

\begin{proof}
In the diagram 
$$
\xymatrix{
(Z\circ A)^{\co}\circ J \ar[r]^-{\iota\circ J}
&  ((Z\circ A)\bullet I)\circ J
\ar@<2pt>[rr]^-{\varphi^0\circ J}
\ar@<-2pt>[rr]_-{\varphi^1\circ J}
&&  (((Z\circ A)\bullet A)\bullet I)\circ J \\
Z\circ J \ar@<2pt>[r]^-{\upsilon}\ar@{->}@/0pc/[u]^{\nu_Z\circ J}
& ((Z\circ A)\bullet I)\circ J \ar@{=}@/0pc/[u]
\ar@<2pt>[rr]^-{\varphi^0\circ J}
\ar@<-2pt>[rr]_-{\varphi^1\circ J}
&&  (((Z\circ A)\bullet A)\bullet I)\circ J \ar@{=}@/0pc/[u]}
$$
in $\mathcal{M}_J$,
the bottom row is an equalizer for any $I$-comodule $Z$ by
Lemma \ref{lem:ZA_coinv}. 
By Corollary \ref{cor:split}, the top row is obtained by applying $H$ to
the equalizer of an $H$-contractible equalizer pair. Since $H$ is 
comonadic, it preserves such equalizers. Thus the top row is an equalizer too.
Recalling from the proof of the Adjoint Lifting Theorem \cite{J. Power}
the construction of the unit $\nu$ of the adjunction $K\dashv N$, the square
on the left of the diagram is seen to commute. Thus the diagram is serially
commutative. This proves that $H(\nu_Z)=\nu_Z\circ J$ is an isomorphism. Since
$H$ is comonadic, it reflects isomorphisms. This proves that the unit $\nu$ of
the adjunction $K\dashv N$ is a natural isomorphism. Hence $K$ is
fully faithful, see e.g. the dual form of \cite[vol. 1 page 114, Proposition
3.4.1]{Bor}.   
\end{proof}

\subsection{The Fundamental Theorem of Hopf modules}

Our final task is to find conditions under which the functor $K$ in
(\ref{eq:triangle}) is an equivalence.

\begin{lemma}\label{lem:theta}
For any bimonoid $A$ in a duoidal category $\mathcal M$, and for any
$I$-comodule $Z$,
$$
\vartheta_{Z}:=\big(
\xymatrix@C=20pt{
Z\circ A\ar[r]^-{\rho\circ A}
& (Z\bullet I)\circ A\ar[r]^-{\cong}
& (Z\bullet I)\circ(J\bullet A)\ar[r]^-{\zeta}
& (Z\circ J)\bullet(I\circ A)\ar[r]^-{\cong} & (Z\circ J)\bullet A}\big)
$$
is equal to $\beta_{Z\circ J}.(\nu_Z\circ A)$. Here $\nu$
denotes the unit of the adjunction
$$
H:=\big(
\xymatrix{
\mathcal{M}^I \ar[r]^-{U^I}
& \mathcal{M}\ar[r]^-{(-)\circ J}
& \mathcal{M}_J}\big)
\dashv
G:=\big(
\xymatrix{
\mathcal{M}_J \ar[r]^-{U_J}
& \mathcal{M}\ar[r]^-{(-)\bullet I}
& \mathcal{M}^I}\big),
$$
$\beta$ is the natural transformation
(\ref{eq:beta^A}) and the $A$-action on $Z\circ J$ is induced by the counit (a
monoid morphism) $\varepsilon:A\to J$.
In particular, if $H$ is fully faithful and (\ref{eq:beta^A}) is a natural
isomorphism, then $\vartheta$ is a natural isomorphism.
\end{lemma}

\begin{proof}
The claim follows 
by counitality of the comultiplication $\Delta:A\to A\bullet A$ and
unitality of the multiplication $\varpi:J\circ J \to J$.
\end{proof}

\begin{theorem}\label{thm:fthm}
Let $\mathcal M$ be a duoidal category in which idempotent morphisms
split and
$
H:=\big(
\xymatrix{
\mathcal{M}^I \ar[r]^-{U^I}
& \mathcal{M}\ar[r]^-{(-)\circ J}
& \mathcal{M}_J}\big)
$
is fully faithful. Then for any bimonoid $A$ in $\mathcal M$, the following
assertions are equivalent. 
\begin{itemize}
\item[{(i)}] $I\to A$ is an $A$-Galois extension.
\item[{(ii)}] The natural transformation $\beta$ in (\ref{eq:beta^A}) is an
isomorphism.
\item[{(iii)}] The functor $K$ in (\ref{eq:triangle}) is an equivalence.
\end{itemize}
\end{theorem}

\begin{proof}
$(i)\Rightarrow (ii)$. This assertion is trivial.

$(ii)\Rightarrow (i)$. $A$ is an $A$-comodule monoid by the first and the third
axioms of a bimonoid in Definition \ref{Def:bimonoid}. The $A$-coinvariant
part of any $A$-Hopf module exists by Corollary \ref{cor:split}. The
coinvariant part of $A$ is $I$, and the corresponding monoid morphism in
Proposition \ref{prop:omega} is the unit $\eta:I\to A$ of the monoid $A$,
by Lemma \ref{lem:AcoA}. Then the coequalizers (\ref{eq:BcoA_product})
are trivial hence they exist and are preserved by any functor --- in
particular by any power of $(-)\circ A:\mathcal M \to \mathcal M$. So if (ii)
holds then $\eta: I\to A$ is an $A$-Galois extension. 

$(iii)\Rightarrow (ii)$. If $K$ in (\ref{eq:triangle}) is an equivalence then
in particular the counit $\epsilon$ of the adjunction $K\dashv N$ is a natural
isomorphism. Thus $\beta$ in (\ref{eq:beta^A}) arises as a composite of
natural isomorphisms  
$\xymatrix@C=15pt{
(Q\bullet I)\circ A\ar[r]^-\cong&
(Q\bullet A)^\co \circ A\ar[r]^-{\epsilon_{Q\bullet A}}&
Q\bullet A}$, cf. \eqref{eq:beta_short}.

$(ii) \Rightarrow (iii)$. Since idempotent morphisms in $\mathcal M$
split, they also split in $\mathcal M^I$. Since $H$ is a left adjoint functor
and fully faithful (hence separable, in particular) by assumption, it
is comonadic by \cite[Proposition 3.16]{Mes}. Thus it follows from
Theorem \ref{thm:ff} that $K$ is fully faithful. So we only need to show that
also the counit
$$
\epsilon_X=\big(
\xymatrix{
X^{\co}\circ A  \ar[r]^-{\iota\circ A}
& (X\bullet I)\circ A \ar[rr]^-{(X\bullet \tau)\circ A}
&& (X\bullet J)\circ A\ar[r]^-{\cong}  & X\circ A\ar[r]^-{\gamma}
& X}\big)
$$
of the adjunction $K\dashv N$ is a natural isomorphism, for any $A$-Hopf
module $X$. 

Since $H$ is comonadic, it follows by Corollary \ref{cor:split} that the image
of (\ref{eq:XcoA}) under $H$ is a contractible (thus absolute) equalizer in
${\mathcal M}_J$ --- hence also in $\mathcal M$. Thus the bottom row in  
$$
\xymatrix{
X^{\co}\circ A  \ar[d]_-{\vartheta_{X^{\co},J}}\ar[r]^-{\iota\circ A}
&  (X\bullet I)\circ A
\ar[d]_-{\vartheta_{X\bullet I}}
\ar@<2pt>[rr]^-{\varphi^0\circ A}
\ar@<-2pt>[rr]_-{\varphi^1\circ A}
&&  ((X\bullet A)\bullet I)\circ A
\ar[d]_-{\vartheta_{(X\bullet A)\bullet I}} \\
(X^{\co}\circ J)\bullet A \ar[r]^-{(\iota\circ J)\bullet A }
&  ((X\bullet I)\circ J)\bullet A
\ar@<2pt>[rr]^-{(\varphi^0\circ J)\bullet A}
\ar@<-2pt>[rr]_-{(\varphi^1\circ J)\bullet A}
&&  (((X\bullet A)\bullet I)\circ J)\bullet A}
$$
is an equalizer in ${\mathcal M}$ too. The diagram serially commutes by
the naturality of $\vartheta$ and the vertical arrows are isomorphisms by
Lemma \ref{lem:theta}.
Then the top row is an equalizer too.

The inverse of $\epsilon_X$ is constructed using universality of the equalizer
$$
\xymatrix @R=15pt{
&  X\ar[d]^-{\rho}\ar@{..>}@/0pc/[ldd]_-{\epsilon_X^{-1}} && \\
&X\bullet A\ar[d]^-{\beta_X^{-1}}
\\
X^{\co}\circ A  \ar[r]^-{\iota\circ A}
&  (X\bullet I)\circ A
\ar@<2pt>[rr]^-{\varphi^0\circ A}\ar@<-2pt>[rr]_-{\varphi^1\circ A}
&&  ((X\bullet A)\bullet I)\circ A.}
$$
The vertical arrow was shown to equalize the parallel morphisms in the
bottom row after \eqref{eq:epsinv}. 
Hence there is a unique morphism $\epsilon_X^{-1}:X\rightarrow X^{\co}\circ A$
in $\mathcal M$ such that $(\iota\circ A).\epsilon^{-1}_X=\beta^{-1}_X.\rho$.
The morphism $\epsilon_X.\epsilon_X^{-1}$ is equal to
$$
\xymatrix{
X\ar[r]^-\rho&
X\bullet A\ar[r]^-{\beta_X^{-1}}&
(X\bullet I)\circ A \ar[rr]^-{(X\bullet \tau)\circ A}
&& (X\bullet J)\circ A\ar[r]^-{\cong}  & X\circ A\ar[r]^-{\gamma}
& X.}
$$
This is equal to the identity morphism $X$ by Lemma \ref{Lemma1} and the
counitality of the $A$-coaction on $X$.
Since $\iota\circ A$ is monic, $\epsilon^{-1}.\epsilon$ is equal to the
identity natural transformation $(-)^{\co}\circ A$ if and only if
$(\iota\circ A).\epsilon^{-1}.\epsilon=\iota\circ A$. That is, if and only if
for any Hopf module $X$,
$$
\xymatrix@R=10pt@C=30pt{
\big( X^{\co}\circ A\big.\ar[r]^-{\iota\circ A}
&(X\bullet I)\circ A \ar[r]^-{(X\bullet \tau)\circ A}
& (X\bullet J)\circ A\ar[r]^-{\cong}  & X\circ A\ar[r]^-{\gamma}
& X \ar[r]^-\rho
&\big. X\bullet A\big)=\\
\big(X^{\co}\circ A\big.\ar[r]^-{\iota\circ A}&
(X\bullet I)\circ A \ar[r]^-{\beta_X}&
\big. X\bullet A\big).}
$$
This holds by the compatibility condition between the action and the coaction
on a Hopf module, the equalizer property of $\iota\circ A$,
counitality of the comonoid $I$ and unitality of the monoid $A$, 
as well as the explicit form of $\beta_{X}$ in (\ref{eq:beta^A}).
\end{proof}

\begin{remark}
The equivalent conditions in Theorem \ref{thm:fthm} provide an alternative way
to define a Hopf monoid $A$ in a duoidal category $\mathcal M$. Although we
are not aware of any separating example, the resulting notion does not seem to
be equivalent to any of the definition of a Hopf bimonoid in \cite{T. Booker}
and the property that $(-)\circ A$ is a right Hopf monad on $(\mathcal M,
\bullet)$ (cf. Section \ref{sec:duo_cat}). In fact, it seems to be between
these notions: $(-)\circ A$ is a right Hopf monad provided that $\beta_{Q,M'}$
in \eqref{eq1.6} is an isomorphism, for any right $A$-module $Q$ and any
object $M'$ of $\mathcal M$. The conditions in Theorem \ref{thm:fthm} assert
less: they only say that $\beta_{Q,I}$ is an isomorphism for any right
$A$-module $Q$ and the $\circ$-monoidal unit $I$. The definition of a Hopf
bimonoid in \cite[Definition 9]{T. Booker} requires even less: only
$\beta_{A,I}$ to be an isomorphism.   
\end{remark}

In a braided monoidal category $\mathcal M$, there is only one monoidal unit
$I=J$. Both of its category of modules and comodules are isomorphic to
$\mathcal M$. Thus in this case the functor $H$ in Theorem \ref{thm:fthm} is
an isomorphism. In this sense, Theorem \ref{thm:fthm} extends the Fundamental
Theorem of Hopf Modules in a braided monoidal category with split
idempotents.

More generally, if in a duoidal category $\mathcal M$, the (co)unit
$\tau:I\to J$ is an isomorphism (of monoids and comonoids) then all categories
$\mathcal M^I$, $\mathcal M$ and $\mathcal M_J$ are isomorphic so that the 
functor $H$ in Theorem \ref{thm:fthm} is an isomorphism. Thus in this
case, if idempotent morphisms in $\mathcal M$ split, then all assumptions in
Theorem \ref{thm:fthm} hold.  

\subsection{The dual situation}

Recall from \cite[Section 6.1.2]{M. Aguiar}, that also the opposite of
a duoidal category is duoidal via the roles of the monoidal structures
interchanged. So we can dualize the results in the previous sections without
repeating the proofs. It leads to the following.

\begin{theorem}\label{thm:dual_ff}
Let $\mathcal M$ be a duoidal category in which idempotent morphisms split and
$
G:=\big(
\xymatrix{
\mathcal{M}_J \ar[r]^-{U_J}
& \mathcal{M}\ar[r]^-{(-)\bullet I}
& \mathcal{M}^I}\big)
$
is monadic. Let $A$ be a bimonoid in $\mathcal M$. If 
\begin{equation}\label{eq:varsigman}
\xymatrix{ \varsigma_Q:
Q\circ A \ar[r]^-{\rho\circ A}&
(Q\bullet A)\circ A \ar[r]^-\cong &
(Q\bullet A)\circ (J\bullet A)\ar[r]^-{\zeta}&
(Q\circ J)\bullet (A\circ A)\ar[r]^-{(Q\circ J)\bullet\mu}&
(Q\circ J)\bullet A}
\end{equation}
is an isomorphism for any $A$-comodule $(Q,\rho)$, then the comparison functor
$(-)\bullet A: \mathcal M_J\to \mathcal M_A^A$ is fully faithful.
\end{theorem}

\begin{theorem}\label{thm:dual_fthm}
Let $\mathcal M$ be a duoidal category in which idempotent morphisms
split and
$
G:=\big(
\xymatrix{
\mathcal{M}_J \ar[r]^-{U_J}
& \mathcal{M}\ar[r]^-{(-)\bullet I}
& \mathcal{M}^I}\big)
$
is fully faithful. Then for any bimonoid $A$ in $\mathcal M$, the following
assertions are equivalent. 
\begin{itemize}
\item[{(i)}] The natural transformation $\varsigma$ 
in \eqref{eq:varsigman} is an isomorphism.
\item[{(ii)}] The comparison functor $(-)\bullet A: \mathcal M_J\to \mathcal
M_A^A$ is an equivalence.
\end{itemize}
\end{theorem}

\section{Applications and examples}

In this section we apply Theorem \ref{thm:fthm} and
Theorem \ref{thm:dual_fthm} to the duoidal categories in \cite[Example
6.17]{M. Aguiar} and in \cite[Example 6.18]{M. Aguiar}, respectively.
In particular, we prove that the assumptions of these theorems hold in the
respective examples.

\subsection{The occurrence of idempotent monads}

In the examples in the forthcoming sections, we will work with duoidal
categories in which the $\circ$-monoidal unit $I$ induces an idempotent
comonad $(-)\bullet I$, or the $\bullet$-monoidal unit $J$ induces an
idempotent monad $(-)\circ J$. Therefore in this section we collect some facts
about idempotent (co)monads for later application. As a more general
reference, we recommend \cite[vol. 2 page 196]{Bor}. 

\begin{proposition}\label{prop:idemp_ff}
For a  duoidal category $\mathcal M$ in which the comultiplication
$\delta:I \to I\bullet I$ on the $\circ$-monoidal unit $I$ is an isomorphism,
the following assertions are equivalent.
\begin{itemize}
\item[{(i)}] The functor
$
H:=\big(
\xymatrix{
\mathcal{M}^I \ar[r]^-{U^I}
& \mathcal{M}\ar[r]^-{(-)\circ J}
& \mathcal{M}_J}\big)
$
is fully faithful.
\item[{(ii)}] For any object $M$ of $\mathcal M$ such that
$M\bullet \tau$ is an isomorphism, also $(M\circ \tau)\bullet I$ is an
isomorphism (where $\tau$ is the (co)unit $I\to J$).
\end{itemize}

\end{proposition}
\begin{proof} Since $\delta$ is an isomorphism, $(-)\bullet I:\mathcal
M\rightarrow \mathcal M$ is an idempotent comonad. 
So $\mathcal M^I$ is identified with the full subcategory of
$\mathcal M$ whose objects are those objects $M$ for which the counit 
$$
\xymatrix{
M\bullet I \ar[r]^-{M\bullet \tau}
&M\bullet J \ar[r]^-\cong 
&M}
$$
is an isomorphism, equivalently, $M\bullet \tau$ is an isomorphism.

By the dual form of \cite[vol. 1 page 114, Proposition 3.4.1]{Bor}, the
functor $H$ is fully faithful if and only if the unit  
$$
\nu_M=\big(
\xymatrix@C=15pt{
M\ar[r]^-\cong
&M\bullet J\ar[rr]^-{(M\bullet \tau)^{-1}} 
&&M\bullet I \ar[r]^-{\cong} 
& (M\circ I)\bullet I\ar[rr]^-{(M\circ\tau)\bullet I} 
&& (M\circ J)\bullet I} \big)
$$
of the adjunction 
$
H\dashv \big(
\xymatrix@C=20pt{
\mathcal M_J \ar[r]^-{\,U_J\,}
&\mathcal M\ar[r]^-{\, (-)\bullet I\,}
&\mathcal M^I}\big)
$ 
is an isomorphism, for any object $M$ in $\mathcal M^I$; that is, for any
object $M$ in $\mathcal M$ such that $M\bullet \tau$ is an isomorphism.  
This morphism $\nu_M$ is an isomorphism if and only if $(M\circ\tau)\bullet I$
is an isomorphism.  
\end{proof}

\subsection{The category of spans}

In this section we analyze in some detail the duoidal category
$\mathsf{span}(X)$ of spans over a given set $X$. This duoidal category was
introduced in \cite[Example 6.17]{M. Aguiar}, where it was called the {\em
``category of directed graphs with vertex set $X$''}. 

The objects of $\mathsf{span}(X)$ are triples $(M,t,s)$, where $M$ is a set
and $s$ and $t$ are maps $M\to X$, called the {\em source} and {\em target}
maps, respectively. The morphisms in $\mathsf{span}(X)$ are maps $f:M\to M'$
such that $s'.f=s$ and $t'.f=t$.  

For any spans $M$ and $N$ over $X$, one monoidal structure is given by the
pullback 
$$
M\circ N=\{(m,n)\in M\times N\ |\  s(m)=t(n)\}
\quad \textrm{and}\quad  I=X
$$
and the other monoidal structure is 
$$
M\bullet N=\{(m,n)\in M\times N\ |\ s(m)=s(n),t(m)=t(n)\}
\quad \textrm{and}\quad  J=X\times X.
$$
The interchange law takes the form   
$$
\zeta:(M\bullet N)\circ (M'\bullet N')\rightarrow (M\circ M')\bullet(N\circ
N'), \qquad (m,n,m',n')\mapsto (m,m',n,n').
$$
The $\circ$-monoidal unit $I$ is a comonoid with respect to $\bullet$ via the
comultiplication 
$$
\delta:I\rightarrow I\bullet I=\{(x,y)\in X\times X\ |\ x=y\}\cong I,\qquad
x\mapsto (x,x)\cong x.
$$
The $\bullet$-monoidal unit $J$ is a monoid with respect to $\circ$ via the
multiplication
$$\varpi:J\circ J=\{(x,y,x',y')\ |\ y=x'\}\rightarrow J,\qquad
(x,y=x',y')\mapsto (x,y').
$$
The counit of the comonoid $I$ and the unit of the monoid $J$ are both given
by 
$$
\tau:I\rightarrow J,\qquad x\mapsto (x,x).
$$
The monad $(-)\circ J$ and the comonad $(-)\bullet I$ on $\mathsf{span}(X)$
have the respective object maps
$$
M\circ J
\cong M\times X 
\quad \textrm{and}\quad 
M\bullet I
\cong \{m\in M\ |\ s(m)=t(m)\}.
$$

Let us turn to showing that $\mathsf{span}(X)$ satisfies all assumptions made
on a duoidal category in Theorem \ref{thm:fthm}. 

Since $(-)\bullet I$ is an idempotent comonad on $\mathsf{span}(X)$, its
category of comodules is isomorphic to the full subcategory of
$\mathsf{span}(X)$ whose objects are those spans $(Z,s,t)$ for which
the counit
$
\xymatrix@C=15pt{
Z\bullet I \ar[r]^-{Z\bullet\tau}
&Z\bullet J \ar[r]^-\cong 
&Z}
$
is an isomorphism. An equivalent description of $\mathsf{span}(X)^I$ is the
following.  

\begin{lemma} \label{lem:spanI}
The category $\mathsf{span}(X)^I$ of $I$-comodules is isomorphic to the slice
category $\mathsf{set}/X$ regarded as the full subcategory of
$\mathsf{span}(X)$ whose objects are those spans $(Z,s,t)$ for which $s=t$. 
\end{lemma}
\begin{proof}
For any span $Z$ over $X$, the map
$$
\xymatrix{
\{z\in Z\ |\ s(z)=t(z)\}\ \ar[r]^-\cong
&Z\bullet I \ar[r]^-{Z\bullet\tau}
&Z\bullet J \ar[r]^-\cong 
&Z}
$$
is just the inclusion map, what proves that $Z\bullet\tau$ is an isomorphism
if and only if the source and target maps on $Z$ are equal.
\end{proof}

\begin{proposition}\label{prop:tau_iso}
For any span $Z$ over $X$ with equal source and target maps, 
$(Z\circ \tau)\bullet I$ is an isomorphism.
\end{proposition}

\begin{proof}
For any span $Z$ over $X$, the map
$$
\xymatrix{
\{z\in Z\ |\ s(z)=t(z)\}\cong Z \bullet I \ar[r]^-\cong
&(Z\circ I)\bullet I\ar[r]^-{(Z\circ \tau)\bullet I}
&(Z\circ J)\bullet I \cong
Z}
$$
is again the inclusion map. Hence it is an isomorphism whenever the source and
target maps of $Z$ are equal.
\end{proof}

\begin{proposition}\label{prop:span_split}
Idempotent morphisms in $\mathsf{span}(X)^I$ split.
\end{proposition}

\begin{proof}
For any idempotent morphism $e:M\to M$ in $\mathsf{span}(X)$, also
$\mathsf{Im}(e):=\{e(m)\ |\ m\in M\}$ is a span over $X$ via the restrictions
of the source and target maps of $M$. Hence the epimorphism
$M\to \mathsf{Im}(e)$, $m\mapsto e(m)$ and the monomorphism $\mathsf{Im}(e)\to
M$, $e(m)\mapsto e(m)$ provide a splitting of $e$ in $\mathsf{span}(X)$. This
proves that idempotent morphisms split in $\mathsf{span}(X)$ hence they also
split in the full subcategory $\mathsf{span}(X)^I\cong\mathsf{set}/X$,
cf. Lemma \ref{lem:spanI}.
\end{proof}
 
From Proposition \ref{prop:tau_iso} and Proposition \ref{prop:idemp_ff} we
conclude that the functor $H$ in Theorem \ref{thm:fthm} is fully faithful. So
taking into account also Proposition \ref{prop:span_split}, we see that
Theorem \ref{thm:fthm} holds in the duoidal category $\mathsf{span}(X)$. Our
next task is to identify its bimonoids for which the canonical comonad
morphism (\ref{eq:beta^A}) is a natural isomorphism.    

Recall from \cite[Example 6.43]{M. Aguiar} that a monoid in
$(\mathsf{span}(X), \circ, I)$ is precisely a small category with
object set $X$. 
So is a bimonoid in $\mathsf{span}(X)$ since the monoidal product $\bullet$ is
the categorical product.  
On any elements $a$ and $b$ of a (bi)monoid such that $s(a)=t(b)$, we denote
the multiplication by $\mu(a,b)=:a.b$.  

A right module over a bimonoid $A$ in $\mathsf{span}(X)$ is a span $Q$ over
$X$ equipped with a map of spans $Q\circ A=\{(q,a)\ |\ s(q)=t(a)\}\to Q$,
$(q,a)\mapsto q.a$ which is associative and unital in the evident sense. 
The natural transformation (\ref{eq:beta^A}) takes the explicit form
\begin{eqnarray}\label{eq:span_beta}
\beta_Q:(Q\bullet I)\circ A\cong\{(q,a)\ |\ s(q)=t(q)=t(a) \} 
&\to& Q\bullet A\cong\{(q,a)\ |\ s(q)=s(a),t(q)=t(a)\},\nonumber\\
(q,a) \ \, &\mapsto& (q. a,a).
\end{eqnarray}

\begin{proposition} \label{prop:grp}
Let $A$ be a bimonoid in the duoidal category 
$\mathsf{span}(X)$; that is a small category with object set $X$. 
The corresponding canonical comonad morphism (\ref{eq:span_beta})
is an isomorphism if and only if every element in the monoid $A$
is invertible; that is, $A$ is a groupoid.
\end{proposition}
\begin{proof} If every element in $A$ is invertible, then we construct the
inverse of (\ref{eq:span_beta}) as  
$$
Q\bullet A\rightarrow (Q\bullet I)\circ A,\qquad 
(q,a)\mapsto (q.a^{-1},a).
$$

Conversely, assume that (\ref{eq:span_beta}) is a natural isomorphism. Then it
is an isomorphism, in particular, for $Q=A$. 
Taking into account the explicit form of (\ref{eq:span_beta}), the inverse of
$\beta_A$ can be written as $\beta_A^{-1}(b,a):=(b^a,a)$, in terms of some
function $b^a$ of $a$ and $b$ (such that $s(a)=s(b)$ and $t(a)=t(b)$),
satisfying the conditions 
\begin{equation}\label{1}
\begin{array}{ll}
b^a.a=b,&
\qquad \textrm{for}\   a,b\in A\ \textrm{such that}\ s(a)=s(b),\ t(a)=t(b),\\
(b.a)^a=b,&
\qquad \textrm{for}\   a,b\in A\ \textrm{such that}\ t(b)=s(b)=t(a).
\end{array}
\end{equation}

Introducing the notation $M_{x,y}=\{m\in M |\ s(m)=x,t(m)=y\}$, for any span
$M$ over $X$, $\beta_A$ induces bijections 
$$
(\beta_A)_{x,y}:((A\bullet I)\circ A)_{x,y}\rightarrow (A\bullet
A)_{x,y},\qquad 
(b,a)\mapsto \ \ (b.a,a).
$$
Any element $c\in A$ such that $s(c)=t(c)=y$, induces two maps 
$$
\begin{array}{ll}
\varphi_c:((A\bullet I)\circ A)_{x,y}\rightarrow ((A\bullet I)\circ A)_{x,y},
\qquad 
&(b,a)\mapsto (c.b,a)\\
\psi_c:(A\bullet A)_{x,y}\rightarrow (A\bullet A)_{x,y},\qquad
&(b,a)\mapsto (c.b,a)
\end{array}
$$
rendering  commutative the diagram
$$
\xymatrix{ 
((A\bullet I)\circ A)_{x,y}\ar[rr]^-{(\beta_A)_{x,y}}\ar[d]_-{\varphi_c}
&& (A\bullet A)_{x,y}\ar[d]^-{\psi_c} \\
((A\bullet I)\circ A)_{x,y}\ar[rr]_-{(\beta_A)_{x,y}}  && (A\bullet A)_{x,y}.}
$$
Equivalently, inverting the horizontal arrows, 
\begin{equation}\label{3}
c.b^a=(c.b)^a,
\qquad \textrm{for}\   a,b,c\in A\ \textrm{such that}\ s(a)=s(b),\ 
t(a)=t(b)=s(c)=t(c).
\end{equation}

For any $x\in X$, denote by $1_x$ the unit morphism at $x$; i.e. the image of
$x$ under the unit $\eta:I=X\to A$. 
For any $c\in A$ such that $s(c)=t(c)=x$, it follows by the first condition in 
(\ref{1}) that $(1_x)^c.c=1_x$. By by (\ref{3}) and the second condition in
(\ref{1}), also $c.(1_x)^c=c^c=1_x$. So $c$ is invertible with the inverse
$(1_x)^c$. 
 
Next we show any morphism from $x$ to $y$ --- i.e. any $a\in A$ such that
$s(a)=x$ and $t(a)=y$ --- is invertible whenever the set $A_{y,x}$ is
non-empty; i.e. there is at least one arrow from $y$ to $x$. Indeed, take  
$a\in A_{x,y}$ and $b\in A_{y,x}$.
Then $a.b\in A_{y,y}$ and $b.a\in A_{x,x}$ are invertible by the previous
paragraph; i.e. $(b.a)^{-1}.b.a=1_x$ and $a.b.(a.b)^{-1}=1_y$. This implies
that $a$ is invertible with the inverse $(b.a)^{-1}.b=b.(a.b)^{-1}$. 

Thus the proof is completed if we show that whenever $A_{x,y}$ is a non-empty
set then also $A_{y,x}$ must be non-empty. Equivalently, if we show that
whenever $A_{x,y}$ is an empty set then also $A_{y,x}$ must be empty. 
Assuming that $A_{x,y}=\emptyset$ for some $x\neq y\in X$, below we construct
an appropriate $A$-module $Q$ such that the corresponding map $\beta_Q$
in (\ref{eq:span_beta}) has a non-trivial kernel unless $A_{y,x}=\emptyset$.   

Fix $x,y\in X$ such that $A_{x,y}=\emptyset$.
Take $Q$ to be the span consisting of two arrows from $u$ to $x$ if $A_{x,u}$
is non-empty and one arrow from $u$ to $x$ if $A_{x,u}$ is empty. That is,
$$
Q:=\{u\stackrel{q_u}\longrightarrow x, u\stackrel{p_u}\longrightarrow x\ |
\ u\in X, A_{x,u}\neq\emptyset\}\cup 
\{u\stackrel{r_u}\longrightarrow x\ |\ 
u\in X, A_{x,u}=\emptyset\}.
$$
Note that if $A_{x,s(a)}$ is non-empty for some $a\in A$, then also
$A_{x,t(a)}$ is non-empty (an element is obtained by composing with $a$).
An associative and unital $A$-action on $Q$ is defined by the prescriptions
$$
\begin{array}{lll}
q_{t(a)}. a=q_{s(a)} \quad \textrm{and} \quad 
&p_{t(a)}. a=p_{s(a)}\quad
&\textrm{if}\ A_{x,s(a)}\neq \emptyset\ \textrm{and}\
A_{x,t(a)}\neq \emptyset\\
q_{t(a)}. a=r_{s(a)} \quad \textrm{and} \quad 
&p_{t(a)}. a=r_{s(a)}\quad
&\textrm{if}\  A_{x,s(a)}=\emptyset \ \textrm{and}\ A_{x,t(a)}\neq \emptyset\\
r_{t(a)}. a=r_{s(a)}
&
&\textrm{if}\  A_{x,s(a)}=\emptyset \ \textrm{and}\  A_{x,t(a)}=\emptyset.
\end{array}
$$
The set $A_{x,x}$ is non-empty since it contains at least the unit arrow
$1_x$. Hence there are two different elements $p_x$ and $q_x$ in $Q$.
If there is at least one element $b$ in $A_{y,x}$, then it obeys
$$
\beta_Q(p_x,b)=(p_x. b, b)=(r_y, b)=(q_x. b, b)=\beta_Q(q_x,b).
$$
Thus $\beta_Q$ has a non-trivial kernel whenever $A_{y,x}$ is non-empty; which
contradicts the assumption that $\beta_Q$ is an isomorphism.
So we proved that $A_{y,x}$ is an empty set whenever $A_{x,y}$ is empty.
\end{proof}

Owing to the fact that the monoidal product $\bullet$ is the categorical
product, a comodule for a comonoid $A$ in $\mathsf{span}(X)$ can be
described as a span $P$ over $X$ equipped with a map of spans $c:P\to A$. The
corresponding coaction sends $p\in P$ to $(p,c(p))$. A morphism of
$A$-comodules is a map of spans $f:P\to P'$ such that $c^\prime.f=c$. 

A Hopf module over a bimonoid $A$ in $\mathsf{span}(X)$ --- that is, over a
small category $A$ with object set $X$ --- is an $A$-module $Q$ equipped
with a morphism of $A$-modules $c:Q\to A$. A morphism of $A$-Hopf
modules is a map of $A$-modules $f:Q\to Q'$  such that $c^\prime.f=c$.  

From Theorem \ref{thm:fthm} and Proposition \ref{prop:grp}, we obtain the
following.  

\begin{corollary}
For a small category $A$ with object set $X$, the following assertions are
equivalent. 
\begin{itemize}
\item[{(i)}] $A$ is a groupoid.
\item[{(ii)}] The natural transformation in (\ref{eq:span_beta}) is an
isomorphism. 
\item[{(iii)}] The canonical comparison functor --- from the slice category
$\mathsf{set}/X$ to the category of $A$-Hopf modules  --- is an equivalence. 
\end{itemize}
\end{corollary}

\subsection{The category of bimodules}

Let $k$ be a commutative, associative and unital ring. Throughout the
section, the unadorned symbol $\otimes$ denotes the $k$-module tensor
product. 
Let $R$ be a commutative, associative and unital $k$-algebra. Its
multiplication will be denoted by juxtaposition on the elements. Denote
by $\mathsf{bim}(R)$ the category of $R$-bimodules. In \cite[ Example
6.18]{M. Aguiar}, it was shown to carry a duoidal structure as follows. For
any $R$-bimodules $M$ and $N$, one of the monoidal structures is provided by
the usual $R$-bimodule tensor product
$$
M \bullet N :=
M\otimes N /\{m\cdot r\otimes n-m\otimes r\cdot n\ |\ r\in R\}
\quad \textrm{and}\quad J=R.
$$
The other monoidal structure is given by an $R\otimes R$-bimodule tensor
product 
$$
M\circ N:=
M\otimes N/\{r\cdot m\cdot r'\otimes n - m\otimes r\cdot n\cdot r'\  |\
r,r'\in R\} \quad \textrm{and}\quad I=R\otimes R.
$$
The interchange law has the form
$$
\zeta:(M\bullet N)\circ (M'\bullet N')\rightarrow 
(M\circ M')\bullet(N\circ N'),\qquad
(m\bullet n)\circ (m'\bullet n')\mapsto 
(m\circ m')\bullet(n\circ n').
$$
The $\circ$-monoidal unit $I$ is a comonoid with respect to $\bullet$ via the
comultiplication 
$$
\delta:I\rightarrow I\bullet I,\qquad  
x\otimes y\mapsto (x\otimes 1_R)\bullet(1_R\otimes y).
$$The $\bullet$-monoidal unit $J$ is a monoid with respect to $\circ$ via the
multiplication 
$$
\varpi:J\circ J\rightarrow J,\qquad 
a\circ b\mapsto ab.
$$
The counit of the comonoid $I$, and the unit of the monoid $J$ are both given
by 
$$
\tau:I\rightarrow J,\qquad  
a\otimes b\mapsto ab.
$$
The comonad $(-)\bullet I$ and the monad $(-)\circ J$ on $\mathsf{bim}(R)$
have the respective object maps
$$
M\bullet I\cong M \otimes R
\quad \textrm{and}\quad
M\circ J\cong M/[M,R].
$$
The isomorphism 
$$
M\circ J\cong M\otimes R / \{ x\cdot m\cdot y \otimes r-m\otimes xry\
|\ x,y\in R \}\cong M/[M,R]=M/\{m\cdot r-r\cdot m\ |\ r\in R\}
$$ 
is established by the mutually inverse maps 
$$
\begin{array}{lrl}
M\circ J\rightarrow M/[M,R], \qquad
&m\circ r=m\cdot r\circ 1_R=r\cdot m\circ 1_R
\mapsto& [r\cdot m]=[m\cdot r]\quad \textrm{and}\\
M/[M,R] \rightarrow  M\circ J,\qquad
&{[m]}\mapsto & m\circ 1_R.
\end{array}
$$
In particular, $J\circ J\cong R/[R,R]\cong R=J$, via the isomorphism provided
by the multiplication $\varpi$ and its inverse $\varpi^{-1}:r\mapsto r\circ 
1_R=1_R\circ r$. 
Thus the monad $(-)\circ J$ on $\mathsf{bim}(R)$ is idempotent. So the
category $\mathsf{bim}(R)_J$ of its modules is isomorphic to the full
subcategory of $\mathsf{bim}(R)$ whose objects are those $R$-bimodules $M$ for
which the unit 
\begin{equation}\label{eq:unit_J}
\xymatrix{
M\ar[r]^-\cong
&M\circ I \ar[r]^-{M\circ \tau}
&M\circ J}
\end{equation}
is an isomorphism. 
Another equivalent description of $J$-modules can be given as follows.

\begin{lemma}\label{lem:diag}
The category $\mathsf{bim}(R)_J$ of $J$-modules is isomorphic to the category 
$\mathsf{mod}(R)$ of $R$-modules --- regarded as the full subcategory of
$\mathsf{bim}(R)$ on whose objects the left and right $R$-actions coincide. 
\end{lemma}
\begin{proof}
For any $R$-bimodule $M$, the map $M\to M\circ J\cong M/[M,R]$ in
(\ref{eq:unit_J}) is the canonical projection. It is an isomorphism if and
only if $[M,R]=0$; that is, the left and right $R$-actions on $M$ coincide.
\end{proof}

In what follows, we check that the assumptions of Theorem \ref{thm:dual_fthm}
hold in $\mathsf{bim}(R)$.   

\begin{proposition}\label{prop:bim_tau_iso}
If $M\circ \tau$ is an isomorphism for some $R$-bimodule $M$, then
$(M\bullet \tau)\circ J$ is an isomorphism too.
\end{proposition}

\begin{proof}
For any $R$-bimodule $M$, the map $(M\bullet \tau)\circ J$ is an isomorphism
if and only if 
$$
\xymatrix@C=12pt{
M\ar[r]^-\cong
&M\otimes R/[M\otimes R,R]\ar[r]^-\cong
&(M\bullet I)\circ J\ar[rr]^-{(M\bullet \tau)\circ J}
&&(M\bullet J)\circ J\ar[r]^-\cong 
&M\circ J\ar[r]^-\cong
&M/[M,R]}
$$
is an isomorphism. The first isomorphism is established by the mutually
inverse maps $M\to M\otimes R/[M\otimes R,R]$, $m\mapsto [m\otimes 1_R]$ and
$M\otimes R/[M\otimes R,R] \to M$, $[m\otimes r]\mapsto r\cdot m$.
The displayed map is the canonical projection. So it is an isomorphism if and
only if $[M,R]=0$. Equivalently, by Lemma \ref{lem:diag}, if and
only if $M\circ \tau$ is an isomorphism.
\end{proof}

Idempotent morphisms in any module category split (through the image). We
conclude by Proposition \ref{prop:bim_tau_iso} and the dual form of
Proposition \ref{prop:idemp_ff} that the functor $G$ in Theorem
\ref{thm:dual_fthm} is fully faithful. So we can apply Theorem
\ref{thm:dual_fthm} to the duoidal category $\mathsf{bim}(R)$. Our next task
is to identify those bimonoids $A$ in $\mathsf{bim}(R)$ for which the
canonical monad morphism $\varsigma$ in \eqref{eq:varsigman} is a natural
isomorphism.    

Recall from \cite[Example 6.44]{M. Aguiar} that a monoid $A$ in
$\mathsf{bim}(R)$ can be described equivalently as a $k$-algebra $A$ equipped
with algebra homomorphisms $s$ and $t$ from $R$ to the center of $A$. (The
algebra homomorphisms $s$ and $t$ are related to the unit $\eta: I=
R\otimes R\to A$ by $s=\eta(-\otimes 1_R)$ and $t=\eta(1_R\otimes -)$.) The
left and right $R$-actions on $A$ come out as 
\begin{equation}\label{eq:center}
r\cdot a=s(r)a=as(r)
\quad\textrm{and}\quad 
a\cdot r=t(r)a=at(r).
\end{equation}

A comonoid in $\mathsf{bim}(R)$ is the usual notion of $R$-coring; that is,
an $R$-bimodule $A$ equipped with a coassociative comultiplication
$\Delta:A\to A\bullet A$ with counit $\varepsilon:A\to R$, such that both the
comultiplication and the counit are $R$-bimodule maps. 
For the comultiplication $A\to A\bullet A$ we use a Sweedler type
index notation $a\mapsto a_1\bullet a_2$, where implicit summation is
understood.

Finally, a bimonoid $A$ in $\mathsf{bim}(R)$ is precisely an
$R$-bialgebroid --- called a ``$\times_R$-bialgebra'' in \cite{M. Takeuchi}
--- whose unit maps $s$ and $t$ land in the center of $A$. Explicitly,
it obeys the following axioms (see \cite[Appendix A1]{C. Ravenel} for the case
when also $A$ is a commutative algebra). 
\begin{itemize}
\item $A$ is a $k$-algebra equipped with algebra homomorphisms $s$ and $t$
from $R$ to the  center of $A$,
\item the $R$-bimodule (\ref{eq:center}) carries an $R$-coring structure,
\item the comultiplication $\Delta:A\to A\bullet A$ and the counit
$\varepsilon:A\to R$ are algebra homomorphisms. 
\end{itemize}

A right comodule over a bimonoid $A$ in $\mathsf{bim}(R)$ is an $R$-bimodule
$Q$ equipped with a coassociative and counital coaction $Q\to Q\bullet A$
which is a morphism of $R$-bimodules. For the coaction $Q\to Q\bullet A$ we
use a Sweedler type index notation $q\mapsto q_0\bullet q_1$, where implicit
summation is understood. For any right $A$-comodule $Q$, the natural
transformation $\varsigma$ in \eqref{eq:varsigman} takes the following
explicit form.  
\begin{equation}\label{eq:varsigma}
\varsigma_Q:Q\circ A \to (Q\circ J)\bullet A\cong Q/[Q,R]\bullet A,\qquad 
q\circ a \mapsto [q_0]\bullet q_1a.
\end{equation}
Recall from \cite{G. Bohm} (and the references therein, in
particular \cite{C. Ravenel} in the commutative case) that an
$R$-bialgebroid $A$ as above is said to be a Hopf algebroid --- with left
bialgebroid structure as above and right bialgebroid structure obtained by
interchanging the roles of $s$ and $t$ --- if in addition there exists a
$k$-module map $S:A\to A$ --- called the {\em antipode} --- such that, for all
$a\in A$ and $r\in R$,   
\begin{equation}\label{eq:antipode_ax}
\begin{array}{ll}
S(as(r))=t(r)S(a),\qquad 
&S(t(r)a)=S(a)s(r),\\
a_1S(a_2)=s(\varepsilon(a)),\qquad
&S(a_1)a_2=t(\varepsilon(a)).
\end{array}
\end{equation}

\begin{proposition}\label{prop:hgd}
For a bimonoid $A$ in $\mathsf{bim}(R)$ --- that is, for an $R$-bialgebroid
$A$ such that the images of the unital maps $s$ and $t$ are central in $A$ ---
the natural transformation (\ref{eq:varsigma}) is an isomorphism if and only
if $A$ is a Hopf algebroid. 
\end{proposition}
\begin{proof} 
If $A$ is a Hopf algebroid, then the inverse of (\ref{eq:varsigma}) is given
by 
$$
\varsigma_Q^{-1}:Q/[Q,R]\bullet A\rightarrow Q\circ A,\qquad 
[q]\bullet a\mapsto q_0\circ S(q_1)a.
$$
In order to see that it is well defined, note that --- 
since the $A$-coaction on $Q$ is morphism of $R$-bimodules,
by (\ref{eq:center}) and the first line in (\ref{eq:antipode_ax}), --- 
$(r\cdot q)_0\circ S((r\cdot q)_1)a=
(q\cdot r)_0\circ S((q\cdot r)_1)a=
q_0\circ S(q_1)(r\cdot a)$, 
for all $a\in A$, $q\in Q$ and $r\in R$. Also --- by
(\ref{eq:center}) and the first line in
(\ref{eq:antipode_ax}), --- $q\cdot r\circ S(b)a=q\circ S(r\cdot b)a$,
for all  $a,b\in A$, $q\in Q$ and $r\in R$. 
It is indeed the inverse of (\ref{eq:varsigma}) since
$$
\varsigma_Q^{-1}\varsigma_Q(q\circ a)=
q_0\circ S(q_1)q_2 a=
q_0\circ t(\varepsilon(q_1))a=
q_0\circ a\cdot \varepsilon(q_1)=
q_0\cdot \varepsilon(q_1)\circ a=
q\circ a,
$$
$$
\varsigma_Q\varsigma_Q^{-1}([q]\bullet a)=
[q_0]\bullet q_1S(q_2)a=
[q_0]\bullet s(\varepsilon(q_1))a=
[q_0]\bullet \varepsilon(q_1)\cdot a=
[q_0]\cdot \varepsilon(q_1)\bullet a=
[q_0\cdot \varepsilon(q_1)]\bullet a=
[q]\bullet a. 
$$

Conversely, assume that (\ref{eq:varsigma}) is an isomorphism, for any right
$A$-comodule $Q$. Then it is an isomorphism in particular for $Q=I\bullet
A\cong R\otimes A$ with right $R$-action $(r\otimes a)r'=r\otimes a t(r')$ and
$A$-coaction $r\otimes a\mapsto (r\otimes a_1)\bullet a_2$.
So we obtain an isomorphism 
\begin{equation}\label{eq:varsigmbar}
\xymatrix@C=15pt {
A\star A :=A\otimes A/\{a\cdot r\otimes b-a\otimes b\cdot r\ |\ r\in R\}
\ar[r]^-\cong 
&(I\bullet A)\circ A\ar[r]^-{\varsigma_{I\bullet A}}
&((I\bullet A)\circ J)\bullet A\ar[r]^-\cong
&A\bullet A,}
\end{equation}
to be denoted by $\widehat{\varsigma}$. The first isomorphism in
\eqref{eq:varsigmbar} is established by the mutually inverse maps  
$$
\begin{array}{ll}
A\star A\rightarrow (I\bullet A)\circ A,\qquad 
&a\star b\mapsto ((1_R\otimes 1_R)\bullet a)\circ b,
\quad \textrm{and}\\
(I\bullet A)\circ A \to A\star A,\qquad 
&((r\otimes r') \bullet a)\circ b\mapsto r'\cdot a\star r\cdot b.
\end{array}
$$
The last isomorphism in (\ref{eq:varsigmbar}) is established by the mutually
inverse maps  
$$
\begin{array}{ll}
(I\bullet A)\circ J \cong I\bullet A/[I\bullet A,R] \to A, \qquad
&[(r\otimes r')\bullet a]\mapsto r'\cdot a\cdot r
\quad \textrm{and}\\
A \to (I\bullet A)\circ J \cong I\bullet A/[I\bullet A,R], \qquad
&a\mapsto [(1_R\otimes 1_R)\bullet a].
\end{array}
$$
With these isomorphisms at hand, the explicit form of (\ref{eq:varsigmbar}) is
$\widehat{\varsigma}(a\star b)=a_1\bullet a_2b$, for $a,b\in A$. 
Set $a^+\star a^-:=\widehat{\varsigma}^{\, -1}(a\bullet 1_A)$;
then $\widehat{\varsigma}^{\, -1}(a\bullet b)= a^+\star a^-b$, for all $a,b\in
A$, since $\widehat{\varsigma}$ and thus also its inverse are right $A$-module
maps. Put $S(a):=t(\varepsilon(a^+))a^-$, for all $a\in A$. It is well-defined
since $\varepsilon$ is a right $R$-module map, since $t$ is multiplicative and
by (\ref{eq:center}). We claim that $S$ is an antipode of $A$.

Since $\widehat\varsigma$ is a morphism of right $A$-modules, so is its
inverse. Hence 
$$
(t(r)a)^+\star (t(r)a)^-=
\widehat\varsigma^{\,-1}(t(r)a\bullet 1_A)=
\widehat\varsigma^{\,-1}(a\bullet s(r))=
a^+\star a^-s(r), \qquad \forall a\in A,\ r\in R.
$$
Since the comultiplication on $A$ is a morphism of left $R$-modules, so is
$\widehat\varsigma$. Hence also its inverse is a morphism of left $R$-modules
in the sense that 
$$  
(as(r))^+\star (as(r))^-=a^+s(r)\star a^-, \qquad \forall a\in A,\ r\in R.
$$
With these identities at hand,
\begin{eqnarray*}
S(t(r)a)
&=&t(\varepsilon((t(r)a)^+))(t(r)a)^-
=t(\varepsilon(a^+))a^-s(r)
=S(a)s(r)\qquad \textrm{and}\\
S(as(r))
&=&t(\varepsilon((as(r))^+))(as(r))^-
=t(\varepsilon(a^+s(r)))a^-
=t(r)t(\varepsilon(a^+))a^-
=t(r)S(a),
\end{eqnarray*}
for any $a\in A, r\in R$. The penultimate equality in the second line holds
since $\varepsilon$ is a morphism of left $R$-modules and $t$ is
multiplicative. 

From $\widehat{\varsigma}^{\, -1}.\widehat{\varsigma}=A\star A$, it follows
that 
\begin{equation}\label{eq:multi}
{a_1}^+\star {a_1}^{-}a_2=a\star 1_A,\qquad \forall a\in A.
\end{equation}
Since $\widehat{\varsigma}$ is a left $A$-comodule map; i.e. 
$(\Delta\bullet A).\widehat{\varsigma}=
(A\bullet\widehat{\varsigma}).(\Delta\star A)$, also
$\widehat{\varsigma}^{\, -1}$ is a left $A$-comodule map. That is, 
\begin{equation}\label{eq:comulti} 
{a^+}_1\bullet {a^+}_2\star a^-=
a_1\bullet {a_2}^+\star {a_2}^-,\qquad \forall a\in A.
\end{equation}
Composing both sides of the equality 
$\mu=\big(
\xymatrix@C=12pt{
A\star A\ar[r]^-{\widehat{\varsigma}}
&A\bullet A\ar[r]^-{\varepsilon\bullet A}
&J\bullet A\ar[r]^-\cong
&A}\big)
$ by $\widehat{\varsigma}^{\, -1}$, we obtain
\begin{equation}\label{eq:counit}
a^+a^-=s(\varepsilon(a)),\qquad \forall a\in A.
\end{equation} 
Then by (\ref{eq:multi}) it follows that 
$$
S(a_1)a_2=t(\varepsilon({a_1}^+){a_1}^-a_2=
t(\varepsilon(a))1_A=
t(\varepsilon(a)),$$
and by (\ref{eq:comulti}) and (\ref{eq:counit}),
$$
a_1S(a_2)=
a_1t(\varepsilon({a_2}^+){a_2}^-=
{a^+}_1t(\varepsilon({a^+}_2)a^-
=a^+a^-=
s(\varepsilon(a)),
$$
for any $a\in A$.
This proves that $A$ is a Hopf algebroid.
\end{proof}

A right module over a bimonoid $A$ in $\mathsf{bim}(R)$ is, equivalently, a
right module over the constituent $k$-algebra. It is an $R$-bimodule via the
actions induced by $s$ and $t$. A morphism of $A$-modules in $\mathsf{bim}(R)$
is a morphism of modules over the constituent $k$-algebras; it is
automatically a morphism of $R$-bimodules. 

A right comodule of a bimonoid $A$ in $\mathsf{bim}(R)$ is an $R$-bimodule $Q$ 
equipped with a coassociative and counital coaction $Q\to Q\bullet A$ which is
a morphism of $R$-bimodules. A morphism of $A$-comodules in $\mathsf{bim}(R)$
is an $R$-bimodule map which is compatible with the coactions in the evident
sense. 

A Hopf module $M$ over a bimonoid $A$ in $\mathsf{bim}(R)$ is a right
$A$-module which is also an $A$-comodule via the left and right $R$-actions
induced by $s$ and $t$, respectively; such that the compatibility condition
$(m\cdot a)_0\bullet (m\cdot a)_1=m_0\cdot a_1\bullet m_1a_2$ holds, for all
$m\in M$ and $a\in A$. A morphism of $A$-Hopf modules is a morphism of modules
over the constituent $k$-algebras --- hence also a morphism of $R$-bimodules
--- which is compatible with the coactions in the evident sense.   

From Theorem \ref{thm:dual_fthm} and Proposition \ref{prop:hgd},
we obtain the following.  

\begin{corollary}
Let $R$ be a commutative algebra over a commutative ring $k$. Let $A$ be an
$R$-bialgebroid whose unit $R\otimes R\to A$ takes its values in 
the center of $A$ --- equivalently, let $A$ be a bimonoid in the duoidal
category $\mathsf{bim}(R)$. Then the following assertions are equivalent.
\begin{itemize}
\item[{(i)}] $A$ is a Hopf algebroid (via the given left bialgebroid
structure and the right bialgebroid structure obtained by interchanging the
roles of the unital maps $s$ and $t$).
\item[{(ii)}] The natural transformation (\ref{eq:varsigma}) is an
isomorphism. 
\item[{(iii)}] The canonical comparison functor --- from the category of
$R$-modules to the category of $A$-Hopf modules --- is an equivalence.
\end{itemize}
\end{corollary}

\begin{noteadded}
Soon after we had submitted the first version of this paper (on the 5th of
December in 2012), two closely related papers \cite{AguCha} and \cite{MesWis}
appeared in the arXiv (although \cite{AguCha} was submitted for
publication much earlier). Their relation to our work is analyzed
in \cite{MesWis}, here we shortly recall that on the request of the referee. 

In \cite{AguCha}, Aguiar and Chase study the following situation. They
consider a bimonad $T$ on a monoidal category $\mathcal C$ (which can be taken
to be e.g. the bimonad $(-)\circ A$ induced on $(\mathcal M, \bullet)$ by a
bimonoid $A$ in a duoidal category $\mathcal M$); a $T$-comodule monad $S$
(which can be chosen to be $(-)\circ A$ as well); and an arbitrary comonoid
$c$ in $\mathcal C$ (which can be taken to be e.g. the $\circ$-monoidal unit
$I$ in the duoidal category $\mathcal M$). Associated to these data, there is
a category of generalized Hopf modules (with our choices it comes out as the
category $\mathcal M^A_A$ of Hopf modules in Definition \ref{def:Hopf_mod});
and a comparison functor $K$ from the category of $c$-comodules to this
category of generalized Hopf modules (it reduces to the same functor $K$
in \eqref{eq:triangle} in our case). In \cite[Theorem 5.8]{AguCha}, under
certain assumptions on $S$ and $c$, it is proven that $K$ is an equivalence if
and only if a canonical `Galois morphism' (reducing to \eqref{eq:beta^A} with
our choices) is an isomorphism. 

Comparing the assumptions in \cite[Theorem 5.8]{AguCha} with the dual
forms of the Beck criteria (see \cite[page 100, Theorem 3.14]{M. Barr},
cf. Section \ref{sec:Dubuc-Beck}), they imply in turn that the composite of
the forgetful functor $U_c$ corresponding to the category of $c$-comodules,
and of the left adjoint $F_S$ of the forgetful functor corresponding to the
Eilenberg-Moore category of $S$-modules, is comonadic (whence the conclusion
follows by \cite[Theorem 1.7]{GomTor} or \cite[Theorem 4.4]{Mes}, see Section 
\ref{sec:Dubuc-Beck}). Indeed, $F_SU_c$ is a composite of two left adjoint
functors hence it is left adjoint. Since $S$ is assumed to be conservative, so
is $F_S$ and thus also the composite functor $F_SU_c$. Taking an
$F_SU_c$-contractible equalizer pair $(f,g)$, it is taken by $U_c$ to an
$F_S$-contractible equalizer pair $(U_c f,U_c g)$. By assumption, their
equalizer is created by $F_S$. Moreover, by the assumption that $(-)\otimes c$
and $(-)\otimes c\otimes c$ preserve the equalizer of $(U_cf,U_cg)$, it
follows that $U_c$ creates the equalizer of $(f,g)$.   

Applying \cite[Theorem 5.8]{AguCha} to the comparison functor $K$ in
\eqref{eq:triangle}, this means that {\em assumptions are made on $A$}, which
imply the comonadicity of the functor in the bottom row of
\eqref{eq:triangle}. As a conceptual difference, in Section \ref{sec:fthm} of
this paper we make no assumption on $A$. We prove the comonadicity of the
functor in the bottom row of \eqref{eq:triangle} from assumptions {\em on the
duoidal category $\mathcal M$ alone}. 

In Mesablishvili and Wisbauer's paper \cite{MesWis}, a slight generalization
and an alternative proof of our Theorem \ref{thm:fthm} is presented. In their
generalized version, the functor in the bottom row of the diagram below is
only required to be separable (not necessarily fully faithful). In their proof,
they avoid the explicit construction of the inverse of the comparison functor
$K$ in \eqref{eq:triangle}. Instead, they observe that there is a commutative
diagram
$$
\xymatrix{
\mathcal M^I \ar[r]^-{U^I}\ar@{=}[d]&
\mathcal M\ar[r]^-{(-)\circ A}&
\widetilde {\mathcal M}_A\ar[d]^-{(-)\circ_A J}\\
\mathcal M^I \ar[r]_-{U^I}&
\mathcal M\ar[r]_-{(-)\circ J}&
\mathcal M_J}
$$
where $\widetilde {\mathcal M}_A$ stands for the Kleisli category of the monad
$(-)\circ A$ (i.e. the category of free right $A$-modules). Since the left
adjoint functor in the bottom row is separable, and idempotent morphisms in
$\mathcal M$ (and thus in $\mathcal M^I$) split by assumption, it follows by
\cite[Proposition 1.13]{MesWis} that the functor $U^I(-)\circ A$ in the top
row reflects isomorphisms and any $U^I(-)\circ A$-contractible pair possesses
a contractible equalizer in $\mathcal M^I$. These properties imply that
composing the functor in the top row with the fully faithful embedding 
$\widetilde {\mathcal M}_A \to \mathcal M_A$, we obtain a comonadic functor: 
that in the bottom row of \eqref{eq:triangle}. In light of \cite[Theorem
1.7]{GomTor} or \cite[Theorem 4.4]{Mes} (see Section \ref{sec:Dubuc-Beck}),
this provides an alternative proof of Theorem \ref{thm:fthm} (although not
yielding the explicit form of the inverse of the comparison functor $K$
in \eqref{eq:triangle}). 
\end{noteadded}

\end{document}